\documentclass{article}

\usepackage[english]{babel}

\usepackage[letterpaper,top=2cm,bottom=2cm,left=3cm,right=3cm,marginparwidth=1.75cm]{geometry}

\usepackage{amsmath, amssymb}
\usepackage{graphicx}
\usepackage[colorlinks=true, allcolors=blue]{hyperref}

\usepackage{enumerate} % format of items in 'enumerate'
\usepackage{amsfonts}
\usepackage{amsthm}

\usepackage{xcolor} % colors

\usepackage{soul} % \sp
\usepackage{scrextend}

\newtheorem{theorem}{Theorem}
\newtheorem{lemma}[theorem]{Lemma}
\newtheorem{corollary}[theorem]{Corollary}
\newtheorem{claim}[theorem]{\emph{Claim}}
\newtheorem{definition}[theorem]{Definition}
\newtheorem{observation}[theorem]{Observation}
\newtheorem{conj}[theorem]{Conjecture}
\newtheorem{question}[theorem]{Question}

\newenvironment{subproof}{%
  \begin{addmargin}[1em]{0em}\begin{proof}[Proof of claim]%
}{%
  \end{proof}\end{addmargin}\medskip%
}

\renewcommand{\nsim}{\not\sim} % not adjacency relation
\newcommand{\dom}{\text{dom}}

\newcommand{\numsucc}{\text{succ}}

\newcommand{\ex}[1]{\mathbb{E}[#1]}
\renewcommand{\Pr}[1]{\textup{\textbf{Pr}}\left[#1\right]}

\newcommand{\join}{\vee} % complete joint
\newcommand{\minusjedna}{\deg^-} % degree minus 1
\newcommand{\chidp}[1]{\chi_{DP}(#1)} % correspondence chromatic number
\newcommand{\chidpu}{\chi_{DP}}
\newcommand{\exc}{\text{exc}}

\newcommand{\CC}{\mathcal{C}}
\newcommand{\DD}{\mathcal{D}}
\newcommand{\WW}{\mathcal{W}}
\newcommand{\KK}{\mathcal{K}}
\newcommand{\EE}{\mathcal{E}}

\newcommand{\nat}{\mathbb{N}}

\newcommand{\thebound}{3\cdot 10^9}

\title{On Borodin-Kostochka conjecture for correspondence coloring\footnote{An extended abstract of this paper appeared at Eurocomb'25.}}
\author{Zdeněk Dvořák\thanks{Charles University, Prague, Czech Republic.
	\href{mailto:rakdver@iuuk.mff.cuni.cz}{rakdver@iuuk.mff.cuni.cz}.
	Supported by ERC-CZ project LL2328 (Beyond the Four Color Theorem) of the Ministry of Education of Czech Republic.} 
\and Ross J. Kang\thanks{Korteweg--de Vries Institute for Mathematics, University of Amsterdam, the Netherlands. 
    \href{mailto:r.kang@uva.nl}{r.kang@uva.nl} 
    Partially supported by grant OCENW.M20.009 of the Dutch Research Council (NWO) and the Gravitation Programme NETWORKS (024.002.003) of the Dutch Ministry of Education, Culture and Science (OCW).}
\and David Mikšaník\thanks{Charles University, Prague, Czech Republic.
	\href{mailto:miksanik@iuuk.mff.cuni.cz}{miksanik@iuuk.mff.cuni.cz}.
    Supported by the the project GAUK 40324 of the Charles University Grant Agency.}}
\date{}

\begin{document}
\maketitle

\begin{abstract} 
Borodin and Kostochka in 1977 conjectured that if a graph $G$ has maximum
degree $\Delta(G)\ge 9$ and its clique number satisfies $\omega(G)\le
\Delta(G)-1$, then its chromatic number satisfies $\chi(G) \le \Delta(G)-1$. We
prove this statement with respect to a stronger graph coloring parameter, the
correspondence chromatic number $\chi_{DP}$, provided the maximum degree is
sufficiently large. More precisely, we prove that for every integer $\Delta\ge \thebound$, a graph $G$ of maximum degree
at most $\Delta$ satisfies $\chi_{DP}(G) \le \max(\omega(G),\Delta-1)$.
This strengthens earlier results of Reed (1999) for usual chromatic number
and of Choi, Kierstead and Rabern (2023) for list chromatic number.
\end{abstract}

%===================================================================

% Introduction
%\input{sections/01_eurocomb}
%\input{sections/01}

\section{Introduction} \label{section_01}

Here is a common refrain in combinatorics: \emph{How do local constraints influence global structure?}
A classic result of this flavour in graph coloring is the celebrated theorem of Brooks~\cite{Bro41}, which implies that for any $\Delta\ge 3$
and for any graph $G$ of maximum degree at most $\Delta$, if $\omega(G) \leq \Delta$, then the graph~$G$ is $\Delta$-colorable.\footnote{We use $\delta(G)$,
$\Delta(G)$, $\omega(G)$, and $\chi(G)$ to denote the \emph{minimum degree}, the \emph{maximum degree}, the \emph{clique number}, and the \emph{chromatic number} of the graph $G$,
respectively.}
In other words, a local assumption ($\omega(G) \leq \Delta$) gives a bound on the global structure ($\chi(G) \leq \Delta$)
improved over what holds in general without this assumption ($\chi(G) \leq \Delta + 1$). Since $\omega(G) \leq \chi(G)$, this bound on $\chi(G)$ is the best we can obtain
based only on this assumption.\footnote{We remark that $\Delta\ge 3$ is needed to exclude
odd cycles.  Of course, we could replace this condition by including odd cycles as a special case in the outcome.}

Vizing \cite{Viz68} asked for other extremal bounds on $\chi(G)$ given stronger assumptions on $\omega(G)$.
Already the most restrictive case $\omega(G)\le 2$ (i.e., the case of triangle-free graphs) is rather challenging
and has attracted a lot of attention, see e.g.~\cite{BK77,Cat78,Law78,Joh96+a,Mol19,Ber19,DJKP20,HuPi23,KaRo23+}.
For results on other regimes of $\omega(G)$ in terms of the maximum degree $\Delta(G)$, see e.g.~\cite{Joh96+b,BKNP22,DKPS20+,HJK22}.
Our focus in this paper is the other non-trivial extreme case of Vizing's problem, i.e., the case that $\omega(G) = \Delta(G) - 1$.
The most important topic in this case is Borodin-Kostochka Conjecture~\cite{BK77}.

\begin{conj}[Borodin and Kostochka~\cite{BK77}] \label{conj__bk_conjecture}
    For every graph $G$ such that $\Delta(G) \geq 9$, if $\omega(G) \leq \Delta(G) - 1$, then $\chi(G) \leq \Delta(G) - 1$.
\end{conj}

\noindent
The graph $C_5 \boxtimes K_3$ (see Figure~\ref{fig__strong_products}), i.e., the strong product of the cycle on five vertices with the triangle,
is $8$-regular and has clique number $6$, yet its chromatic number is $8$ (since $|V(G)|=15$ and every independent set in $G$ has size at most two).
This shows that the assumption $\Delta(G) \geq 9$ in this conjecture is necessary.
The most substantial progress towards resolving Borodin-Kostochka Conjecture is due to Reed~\cite{Ree99} who showed that a (much) larger lower bound assumption on $\Delta(G)$ is sufficient.

\begin{figure}
\centering
    \includegraphics[scale=0.7]{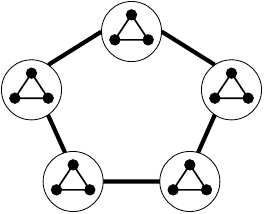}
    \caption{The graph $C_5 \boxtimes K_3$.}
    \label{fig__strong_products}
\end{figure}

\begin{theorem}[Reed~\cite{Ree99}] \label{thm__bk_conjecture_is_true_asym_for_coloring}
    For every graph $G$ such that $\Delta(G) \geq 10^{14}$, if $\omega(G) \leq \Delta(G) - 1$, then $\chi(G) \leq \Delta(G) - 1$.
\end{theorem}

One way to approach Borodin-Kostochka Conjecture is by considering weaker statements.
As an example, Cranston and Rabern~\cite{CR15} showed the conclusion $\chi(G) \leq \Delta(G) - 1$ holds under the assumption that $\Delta(G) \geq 13$ and $\omega(G) \leq \Delta(G) - 3$.
Moreover, the conjecture has been confirmed for graphs from various restricted graph classes, see e.g.~\cite{Dhu21, CLLZ24, CR13_claw-free_graphs} for results of this form.

Another approach is to consider strengthened or more general versions of the conjecture,
e.g., to replace the chromatic number by a stronger coloring parameter.
For instance, one can consider \emph{list coloring}, the notion introduced independently by Erd{\H{o}}s, Rubin, and Taylor~\cite{ERT80} and by Vizing~\cite{Viz76}.
In this setting, each vertex has its own list of $k$ colors and we need to choose its color from its list.
The \emph{list chromatic number} $\chi_\ell(G)$ of a graph $G$ then is the minimum integer $k$ such that~$G$ can be properly colored from any such assignment of lists of size $k$
(see below for a more precise definition).  Clearly the list chromatic number of a graph is at least as large as its ordinary chromatic number,
which corresponds to the special case in which the lists must all be the same.  Thus, it is interesting that an analogue of Theorem~\ref{thm__bk_conjecture_is_true_asym_for_coloring}
holds in this more general setting as well.
\begin{theorem}[Choi, Kierstead, and Rabern~\cite{CKR23}] \label{thm__bk_conjecture_is_true_asym_for_list_coloring}
    For every graph $G$ such that $\Delta(G) \geq 10^{20}$, if $\omega(G) \leq \Delta(G) - 1$, then $\chi_\ell(G) \leq \Delta(G) - 1$.
\end{theorem}

We continue the investigation along these lines by considering an even stronger coloring
parameter, the \emph{correspondence chromatic number} $\chidpu$.
The main idea of correspondence coloring is to additionally generalize the way colors
of different vertices interact.  This notion was introduced by Dvořák and Postle~\cite{DP18} to solve a list
coloring problem for planar graphs.\footnote{Many papers use the terms \emph{DP-coloring} and \emph{DP-chromatic number}
based on their initials; we prefer to use the terminology they introduced, with the exception of the $\chidpu$
notation.}  This notion also sheds light on which coloring techniques and results are ``local'' (only dependent on
the fact that coloring a vertex prevents us from using at most a single specific color at each of its neighbors)
and which depend on a global coherence of the colors.

More precisely, for a graph $G$, a \emph{list assignment} is a function $L \colon V(G) \rightarrow 2^{\text{colors}}$
assigning to each vertex $v$ a list $L(v)$ of colors that it can use.  We say that $L$ is a \emph{$k$-list assignment} if $|L(v)| \geq k$ holds
for every $v \in V(G)$.  An \emph{$L$-coloring} of $G$ is a proper coloring $\psi$ such that $\psi(v)\in L(v)$ holds for every vertex $v\in V(G)$.
We say that the graph $G$ is \emph{$k$-list colorable} if $G$ has an $L$-coloring for every $k$-list assignment $L$;
and the \emph{list chromatic number} $\chi_\ell(G)$ of the graph $G$ is then defined as the minimum integer $k$ such that $G$ is $k$-list colorable.

In correspondence coloring, we additionally specify which colors in lists of adjacent vertices are in conflict.
Formally, a \emph{correspondence assignment} for $G$ is a pair $(L, \Pi)$, where
\begin{itemize}
\item $L$ is a list assignment for~$G$ and
\item $\Pi$ is a function assigning to each edge $e=uv \in E(G)$ a
(not necessarily perfect) matching $\Pi(e)$ between $\{u\} \times L(u)$ and $\{v\} \times L(v)$.
\end{itemize}
That is, for $e=uv \in E(G)$, $\Pi(e)$ is a set containing pairs $\{c,d\}$, where $c\in \{u\} \times L(u)$ and $d\in \{v\} \times L(v)$,
such that each element of $(\{u\} \times L(u))\cup(\{v\} \times L(v))$ belongs to at most one such pair in $\Pi(e)$.
We say that a correspondence assignment $(L, \Pi)$ is a \emph{$k$-correspondence assignment} if $L$ is a $k$-list assignment.
A function $\varphi$ assigning a color to each vertex of $G$ is an \emph{$(L, \Pi)$-coloring} if
\begin{itemize}
\item $\varphi(v) \in L(v)$ holds for every vertex $v \in V(G)$, and
\item every edge $e=uv\in E(G)$ satisfies $\{(u, \varphi(u)), (v, \varphi(v))\} \notin \Pi(e)$.
\end{itemize}
That is, the matching $\Pi(e)$ describes which colors at the endpoints $u$ and $v$ of the edge $e$ conflict, and~$\varphi$ must assign non-conflicting
colors to adjacent vertices.
We say that the graph $G$ is \emph{$(L, \Pi)$-colorable} if there exists an $(L, \Pi)$-coloring of $G$.
Furthermore, we say that $G$ is \emph{$k$-correspondence colorable} if~$G$ is $(L, \Pi)$-colorable for every $k$-correspondence assignment $(L, \Pi)$ of $G$.
Finally, the \emph{correspondence chromatic number} $\chidp{G}$ of $G$ is defined as the smallest integer $k$ such that $G$ is $k$-correspondence colorable.

For a correspondence assignment $(L, \Pi)$ and an edge $e=uv$, we say that $\Pi(e)$ is the \emph{identity matching} if $\Pi(e)=\{\{(u,\alpha),(v,\alpha)\}:\alpha\in L(u)\cap L(v)\}$.
In the special case that $\Pi$ assigns the identity matching to every edge, the $(L, \Pi)$-colorings are exactly the $L$-colorings of the graph.
If $L$ additionally assigns the same list to every vertex,  the $(L, \Pi)$-colorings are ordinary proper colorings.
This implies that
$$\chi(G)\le \chi_\ell(G)\le \chidp{G}$$
holds for every graph $G$.  Many of the techniques used to bound list chromatic number also
apply in the correspondence coloring setting; e.g., by coloring the vertices greedily one by one, it is easy to see that $$\chidp{G} \leq \Delta(G) + 1.$$

On the other hand, these coloring parameters diverge significantly already on bipartite graphs, i.e., for graphs $G$ with $\chi(G)=2$.
It is a classic result of Erd\H{o}s, Rubin, and Taylor~\cite{ERT80} that $\chi_\ell(K_{n,n}) \sim \log_2 n$ as $n\to\infty$,
and Alon and Krivelevich~\cite{AlKr98} conjectured that $\chi_\ell(G) = O(\log \Delta(G))$ holds for any bipartite graph $G$.
In contrast, \cite{KPV05,Ber16} observed that $\chidp{G} = \Omega(\Delta/\log \Delta)$ holds for \emph{every} $\Delta$-regular graph~$G$ (even for bipartite ones).

Thus, it is not {\em prima facie} intuitive if Theorems~\ref{thm__bk_conjecture_is_true_asym_for_coloring} and \ref{thm__bk_conjecture_is_true_asym_for_list_coloring}
should extend to the correspondence coloring setting.  As our main result, we prove that this is indeed the case.

\begin{theorem} \label{thm__bk_conjecture_is_true_asym_for_DP_coloring}
    For every integer $\Delta\ge \thebound$,
    if $G$ is a graph of maximum degree at most $\Delta$ and clique number at most $\Delta-1$,
    then $\chidp{G} \leq \Delta - 1$.
\end{theorem}

This is equivalent to saying that for every graph $G$ such that $\Delta(G) \geq \thebound$, if $\omega(G) \leq \Delta(G) - 1$, then $\chidp{G} \leq \Delta(G) - 1$,
analogously to the statements of Theorems~\ref{thm__bk_conjecture_is_true_asym_for_coloring} and \ref{thm__bk_conjecture_is_true_asym_for_list_coloring}; however, the formulation
from Theorem~\ref{thm__bk_conjecture_is_true_asym_for_DP_coloring} is more convenient for the proof (since the graphs arising in the inductive argument can have maximum degree less than $\Delta(G)$).

Let us remark that the equality $\chi(G)= \chi_\ell(G)= \chidp{G}$ clearly holds for all graphs $G$ with $\omega(G) = \Delta(G) + 1$.
More interestingly, the parameters also coincide for graphs~$G$ with $\omega(G) = \Delta(G) \geq 3$,
by the analogue of Brooks' theorem for correspondence coloring (see~\cite{Zaj18} for a short proof).
We find the following open question to be interesting (Theorem~\ref{thm__bk_conjecture_is_true_asym_for_DP_coloring} gives $\Delta_0\le \thebound$).

\begin{question} \label{q_dp_analogoue_of_dp_coloring}
    What is the smallest number $\Delta_0$ such that every graph $G$ with $\Delta(G) \geq \Delta_0$ and $\omega(G) = \Delta(G) - 1$ satisfies $\chi(G) = \chi_\ell(G) = \chidp{G}$?
\end{question}

Let us now discuss the main ideas behind our proof of Theorem~\ref{thm__bk_conjecture_is_true_asym_for_DP_coloring}.
At a high level, we proceed analogously to the proofs of Theorems~\ref{thm__bk_conjecture_is_true_asym_for_coloring} in~\cite{Ree99} and
Theorem~\ref{thm__bk_conjecture_is_true_asym_for_list_coloring} in~\cite{CKR23}.
We consider a hypothetical minimal counterexample~$G$ and an arbitrary $(\Delta - 1)$-correspondence assignment $(L, \Pi)$,
and show that~$G$ is $(L, \Pi)$-colorable, thereby proving that no counterexample to Theorem~\ref{thm__bk_conjecture_is_true_asym_for_DP_coloring} exists.

To this end, we first use reducible configuration arguments to obtain structural information about~$G$.  Most importantly, we show that for each clique $C$ of size $\Delta-1$,
no vertex of $V(G)\setminus V(C)$ has more than $85$ neighbors in $C$.  In this part, it is actually useful that we work in the correspondence coloring setting; the proof of Claim~\ref{cl__nomatch}
exploits the fact that we can specify correspondences on edges which we add during a reduction.  Thus, this structural result turns out to be somewhat stronger than
the corresponding lemma of Choi, Kierstead, and Rabern~\cite{CKR23}.

The next part is essentially the same as in both~\cite{Ree99} and~\cite{CKR23}:  We use the structural results to show that the graph $G$
has a \emph{sparse-dense decomposition} into large disjoint cliques (or more precisely, \emph{near-cliques} whose non-edges are incident with a single vertex)
and \emph{sparse} vertices with many non-edges among their neighbors.

The last part is a probabilistic argument, where we deviate substantially from both Reed~\cite{Ree99} and Choi, Kierstead, and Rabern~\cite{CKR23}.
In their proofs, this part was accomplished by an analysis of the \emph{naive coloring procedure}: They first assign to each vertex $v \in V(G)$ a color chosen uniformly independently at random from its list $L(v)$.
Of course, this coloring is unlikely to end up being proper.  Hence, they then uncolor all vertices whose color conflicts with the color of at least one of their neighbors; let $\psi$ denote the resulting partial $(L, \Pi)$-coloring of $G$.
Using rather subtle ``expected value and concentration'' arguments,
one then shows that for each uncolored sparse vertex $v$, it is very likely that several of its neighbors retain colors that conflict with the same color
in $L(v)$. This saves a few colors in the list of $v$ and shows that any partial $(L, \Pi)$-coloring which extends $\psi$ can be further greedily extended to $v$.
By similar (but even more technical) arguments, we can show that we also ``save colors'' in near-cliques of the sparse-dense decomposition, and thus this ``greedy extendability''
property holds with high probability separately at each of the near-cliques.
Finally, one uses the Lovász Local Lemma to prove that these favourable circumstances occur at all sparse vertices and near-cliques at the same time
with non-zero probability; and thus we can greedily extend $\psi$ to an $(L, \Pi)$-coloring of the whole graph $G$, by extending it to the sparse vertices and to the near-cliques of the decomposition
one by one.

It is possible to finish the argument in this way in the correspondence coloring setting as well (indeed, we have done so in an earlier version of the paper);
however, the resulting argument is very technical and leads to a rather bad lower bound on $\Delta$ (somewhat worse than the one obtained in~\cite{CKR23}).
Instead, we follow a simpler approach, inspired by~\cite{Ber19}.

In our approach, we first fix an induced subgraph $H$ of $G$ such that for each vertex $v\in V(G)$, roughly $\Delta/5$ neighbors of $v$ belong to $H$ and
roughly 4\% of the non-edges among the neighbors of $v$ have both endpoints in $H$, among several other similar properties.  This is achieved simply by
choosing vertices independently at random with probability roughly $\tfrac{1}{5}$ and applying standard concentration results and Lovász Local Lemma.
We then consider a \emph{uniformly random}
$(L, \Pi)$-coloring $\psi$ of $H$, i.e., each of the $(L, \Pi)$-colorings of $H$ is chosen as $\psi$ with the same probability,
and apply the standard resampling argument to analyze this coloring.
Using the fact that the lists contain $\Delta-1\gg\Delta/5$ colors,
it is fairly easy to show that the coloring $\psi$ almost surely satisfies the ``greedy extendability'' property separately at each of the sparse vertices and near-cliques of the decomposition.

A small issue now is that the distant parts of the uniformly selected coloring $\psi$ are not independent,
and thus we cannot directly apply the standard Lovász Local Lemma to conclude the argument.
However, it is easy to see that they are ``independent enough'' so that the lopsided version of Lovász Local Lemma (Theorem~\ref{tool__lovasz_local_lemma}) applies.

This approach is technically much simpler and completely avoids the use of more involved concentration arguments.
Moreover, it also gives a much better lower bound on $\Delta$.  Indeed, our lower bound $\thebound$ is
better even than the bound given by Reed~\cite{Ree99} for ordinary proper colorings.  Of course, it should
be mentioned that Reed claims that a more careful analysis can bring the bound from his argument to $10^6$ or even less;
but on the other hand, we have also not gone to any great lengths in optimizing our bound.

The rest of the paper is devoted to executing the outlined argument in detail.  In the following section, we introduce the necessary technical definitions and statements.
In Section~\ref{section_03}, we argue that several reducible configurations cannot appear in a minimal counterexample.
In Section~\ref{section_04}, we use this to argue that a minimal counterexample admits a structural decomposition into precisely described ``sparse'' and ``dense'' parts.
Finally, in Section~\ref{section_05}, we use this structural information to prove that the probabilistic approach outlined above results in a correspondence coloring of the whole graph
with non-zero probability.

% Definitions and Preliminaries
\newpage

\section{Definitions and Preliminaries} \label{section_02}

For a positive integer $n$, let $[n]$ denote the set $\{1,2,\ldots,n\}$.
For graphs, we generally use standard notation;  let us point out a few less common or non-standard definitions.
For a graph $G$ and vertices $u,v\in V(G)$, we write $u \sim_G u$ if $u$ is adjacent to $v$ and $u\nsim_G v$ otherwise.
We let $N_G(v)$ denote the set of neighbors of $v$ and $\deg_G v$ their number.
When the graph $G$ is clear from the context, we drop the subscripts.
A vertex of $G$ is \emph{universal} if it is adjacent to all other vertices of $G$.
For another graph $H$, we write
$G \cong H$ to denote that $G$ and $H$ are isomorphic, and $H \leq G$
to denote that $H$ is isomorphic to an induced subgraph of $G$.
The \emph{complement} of $G$, i.e., the graph with the same vertex set in which distinct vertices are adjacent if and only if they are non-adjacent in $G$, is denoted by $\overline{G}$.

For graphs $F$ and $G$, let $F \vee G$ denote the \emph{complete join} of $G$ and $F$, that is, the graph obtained
from the disjoint union of $F$ and $G$ by adding all edges between $V(F)$ and $V(G)$.
Similarly, for a vertex $v \notin V(G)$, let $G \vee v$ denote the graph obtained from $G$ by adding $v$
and edges from $v$ to all vertices of $G$, i.e., so that $v$ is a universal vertex of the resulting graph.

\subsection{Correspondence coloring}

Let us now introduce additional terminology, definitions, and concepts related to correspondence coloring.
Let $(L, \Pi)$ be a correspondence assignment for a graph~$G$.  With a slight abuse of notation,
we also view $(L, \Pi)$ as a correspondence assignment for subgraphs of $G$ (by implicitly restricting $L$
to the vertices and $\Pi$ to the edges of the subgraph).
A \emph{partial $(L, \Pi)$-coloring of $G$} is an $(L, \Pi)$-coloring of an induced subgraph of $G$.

For an edge $e=uv$ of $G$ and colors $\alpha \in L(u)$ and $\beta \in L(v)$, we say that \emph{the color $\alpha$ at $u$ blocks the color $\beta$ at $v$} if $\{(u, \alpha), (v, \beta)\} \in \Pi(e)$.
When the vertex $u$ is clear from the context, we drop the ``at $u$'' part.  For example, if $\psi$ is a partial $(L, \Pi)$-coloring of $G$ 
and a vertex $u$ is assigned a color by $\psi$, then we say ``$\psi(u)$ blocks the color $\beta$ at $v$'' instead of ``$\psi(u)$ at $u$ blocks the color $\beta$ at $v$''.
In particular \emph{$\psi(u)$ blocks no color at $v$} means that the matching $\Pi(e)$ between $\{u\}\times L(u)$ and $\{v\}\times L(v)$ contains no edge incident with $(u,\psi(u))$.
For another edge $e'=u'v \in E(G)$ such that $\psi$ also assigns a color to the vertex~$u'$, we say that \emph{$\psi(u)$ and $\psi(u')$ together block at most one color at $v$}
if $\psi(u)$ blocks no color at $v$, or $\psi(u')$ blocks no color at $v$, or both $\psi(u)$ and $\psi(u')$ block the same color at $v$.

For a partial $(L, \Pi)$-coloring $\psi$ of a graph $G$, let $\dom(\psi)$ denote the set of vertices to which $\psi$ assigns a color.
We define $(L_\psi, \Pi_\psi)$ as the correspondence assignment for the induced subgraph $G-\dom(\psi)$ such that for each vertex $v\in V(G)\setminus \dom(\psi)$,
$$L_\psi(v) = L(v) \setminus \{\alpha\in L(v) : \text{for some vertex $u\in N(v)\cap \dom(\psi)$, the color $\psi(u)$ blocks the color $\alpha$ at $v$}\}$$ 
(that is, we remove from the list of $v$ the colors that conflict with the colors assigned by $\psi$ to the neighbors of $v$),
and for every edge $e=uv\in E(G-\dom(\psi))$,
$$\Pi_\psi(e)=\{\{(u, \alpha), (v, \beta)\} \in \Pi(e): \alpha \in L_\psi(u), \beta \in L_\psi(v)\}.$$
Given a partial $(L, \Pi)$-coloring $\psi'$ of $G$ such that $\dom(\psi)\cap \dom(\psi')=\emptyset$, let $\psi\cup\psi'$
denote the function with domain $\dom(\psi)\cup \dom(\psi')$ which matches $\psi$ on $\dom(\psi)$ and $\psi'$ on $\dom(\psi')$.
Note that $\psi\cup\psi'$ is not in general a partial $(L, \Pi)$-coloring, since the colors on vertices joined by edges between
$\dom(\psi)$ and $\dom(\psi')$ may conflict.  However, the correspondence assignment $(L_\psi, \Pi_\psi)$ is defined so that the following observation holds.
\begin{observation}\label{obs__combine}
Let $(L, \Pi)$ be a correspondence assignment for a graph~$G$ and let $\psi$ and $\psi'$ be partial $(L, \Pi)$-colorings of $G$ such that $\dom(\psi)\cap \dom(\psi')=\emptyset$.
Then $\psi\cup\psi'$ is a partial $(L, \Pi)$-coloring of $G$ if and only if $\psi'$ is a partial $(L_\psi, \Pi_\psi)$-coloring of $G-\dom(\psi)$.
\end{observation}

For a function $f \colon V(G) \rightarrow \nat$, we say that $(L, \Pi)$ is an \emph{$f$-correspondence assignment} if $|L(v)| \geq f(v)$ holds for every $v\in V(G)$;
and that $G$ is \emph{$f$-correspondence colorable} if it is $(L, \Pi)$-colorable for every $f$-correspondence assignment $(L, \Pi)$.
We remark that if $f$ is the constant function assigning to each vertex the value $k$, then $f$-correspondence colorability coincides with $k$-correspondence colorability defined in the introduction.
We say that $(L, \Pi)$ is an \emph{exact $f$-correspondence assignment} if $|L(v)|=f(v)$ holds for every $v\in V(G)$,
and an \emph{exact $k$-correspondence assignment} if $|L(v)|=k$ holds for every $v\in V(G)$.
Of course, to show that the graph $G$ is $f$-correspondence colorable, it suffices to consider exact $f$-correspondence assignments,
since we can simply ignore arbitrarily chosen excess colors at vertices $v$ for which $|L(v)|>f(v)$.

Throughout the paper, we often consider consider $\minusjedna_G$-correspondence assignments, where $\minusjedna_G$ is the
function defined as $\minusjedna_G(v) = \deg_G(v) - 1$ for every $v \in V(G)$.
As usual, when the graph $G$ is clear from the context, we drop the subscript.
Note that a $(\Delta(G) - 1)$-correspondence assignment is also a $\minusjedna$-correspondence assignment, but not necessarily vice versa.
Observation~\ref{obs__combine} together with the following observation show that $\minusjedna$-correspondence assignments are important
in the context of $(\Delta(G)-1)$-correspondence coloring.
\begin{observation}\label{obs__makesminusjedna}
Let $G$ be a graph of maximum degree at most $\Delta$, let $(L, \Pi)$ be a correspondence assignment for $G$, and
let $\psi$ be a partial $(L, \Pi)$-coloring of $G$.
If $(L, \Pi)$ is a $\minusjedna_G$-correspondence assignment (and in particular, if $(L,\Pi)$ is a $(\Delta - 1)$-correspondence assignment),
then $(L_\psi, \Pi_\psi)$ is a $\minusjedna_{G-\dom(\psi)}$-correspondence assignment for $G-\dom(\psi)$.
\end{observation}

We often extend partial colorings greedily, by processing the not yet colored vertices one by one in a fixed order
and giving each of them an arbitrary color not blocked by colors of their previously colored neigbors.
Next, we give a few observations useful in showing that this procedure succeeds, i.e., that in each step the
list of the currently processed vertex contains at least one such unblocked color.
Let $\psi$ be a partial $(L,\Pi)$-coloring of a graph $G$.  For a vertex $v\in V(G)\setminus\dom(\psi)$,
let $$\deg^{-\psi} v=\deg_{G-\dom(\psi)} v=\deg_G v-|N(v)\cap \dom(\psi)|$$ denote the number of not yet colored neighbors of $v$.
We define the \emph{$\psi$-excess} of $v$ as $\exc_\psi(v)=|L_\psi(v)|-\deg^{-\psi} v$.
Note that the greedy coloring procedure used to extend $\psi$ can never get stuck on a vertex $v$ with positive $\psi$-excess:
Indeed, even if some of the $\deg^{-\psi} v$ neighbors of $v$ are given colors during the procedure before $v$,
each of them can block at most one color at $v$ in addition to the colors blocked by $\psi$,
and thus at least $\exc_\psi(v)>0$ colors in $L(v)$ remain unblocked.  This argument also implies the following observation.
\begin{observation}\label{obs__rich}
Let $(L, \Pi)$ be a correspondence assignment for a graph~$G$, let $\psi$ be a partial $(L,\Pi)$-coloring of $G$,
and let $v$ be a vertex not belonging to $\dom(\psi)$.  If $\exc_\psi(v)>0$, then $\psi$ extends to an $(L,\Pi)$-coloring of $G$
if and only if $\psi$ extends to an $(L,\Pi)$-coloring of $G-v$.
\end{observation}

Note that deleting $v$ as in this observation increases the $\psi$-excess of its neighbors by one,
which may make it possible to remove further vertices.
In particular, this has the following consequence, which we often apply in proving reducibility of configurations.

\begin{observation}\label{obs__greedy}
    Let $G$ be a graph, let $(L, \Pi)$ be a correspondence assignment for~$G$, and let $\psi$ be a partial $(L,\Pi)$-coloring of $G$.
    Let $v_1$, \ldots, $v_n$ be an ordering of the vertices of $V(G)\setminus \dom(\psi)$.
    If for each $i\in [n]$ the vertex $v_i$ has at least $1-\exc_\psi(v)$ neighbors in $\{v_{i+1},\ldots,v_n\}$,
    then $\psi$ extends (greedily) to an $(L,\Pi)$-coloring of $G$.
\end{observation}
Let us remark that if $(L, \Pi)$ is a $\minusjedna$-correspondence assignment, then
for every partial $(L,\Pi)$-coloring $\psi$, the $\psi$-excess of every vertex not in $\dom(\psi)$ is at least $-1$.
Thus, in this case the assumption for an index $i\in [n-2]$ is always satisfied if the vertex $v_i$ has at least two neighbors in $\{v_{i+1},\ldots,v_n\}$.
In particular, in our applications of Observation~\ref{obs__greedy}, the vertices $v_1, v_2, \ldots, v_{n-2}$ are often adjacent to both $v_{n-1}$ and $v_n$.
In such cases, we only specify the last two vertices $v_{n-1}$ and $v_n$ of the ordering, as the vertices $v_1, v_2, \ldots, v_{n-2}$ can be ordered arbitrarily.

Furthermore, note that when $(L, \Pi)$ is a $(\Delta(G) - 1)$-correspondence assignment,
it may be easier to use the observation if $G$ has some vertices of degree less than $\Delta(G)$. For example, if $\deg(y) \leq \Delta(G) - 1$, then
$y$ has non-negative $\psi$-excess for every partial $(L, \Pi)$-coloring $\psi$ with $y\not\in\dom(\psi)$.

It is also useful to notice that extending $\psi$ to some of the uncolored vertices cannot decrease the excess of the remaining uncolored vertices.
More precisely, if $\psi'$ is a partial $(L, \Pi)$-coloring of $G$ which extends $\psi$ and $v$ is a vertex not in $\dom(\psi')$, then
$$\exc_{\psi'}(v)\ge \exc_{\psi}(v).$$
Indeed, coloring a neighbor of $v$ may block an additional color at $v$, but it also decreases the number of not yet colored neighbors of $v$.
By extending $\psi$ carefully, we can even increase the excess of $v$.
\begin{observation}\label{obs__improve}
Let $G$ be a graph, let $(L, \Pi)$ be a correspondence assignment for~$G$, and let $\psi$ and $\psi'$ be partial $(L,\Pi)$-colorings of $G$,
where $\psi'$ extends $\psi$.  Let $v$ be a vertex of $G$ not in $\dom(\psi')$.
If
\begin{itemize}
\item $u\in\dom(\psi')\setminus\dom(\psi)$ is a neighbor of $v$ and $\psi'(u)$ blocks no color in $L_\psi(v)$ at $v$, or
\item $u,u'\in\dom(\psi')\setminus\dom(\psi)$ are distinct neighbors of $v$ and $\psi(u)$ and $\psi(u')$ together block at most one color in $L_\psi(v)$ at $v$,
\end{itemize}
then $\exc_{\psi'}(v)>\exc_{\psi(v)}$.
\end{observation}

The following observation is often used to extend $\psi$ to achieve the second option from Observation~\ref{obs__improve}.

\begin{observation} \label{obs__blocking_at_most_one_color}
    Let $G$ be a graph, let $(L, \Pi)$ be a correspondence assignment for $G$, and let $\psi$ be a partial $(L, \Pi)$-coloring.
    Let $v\in V(G)\setminus \dom(\psi)$ be a vertex, let $k$ be an integer such that $\exc_\psi(v)\ge k$, and let $u,u'\in V(G)\setminus \dom(\psi)$ be distinct neighbors of $v$.
    If $u \nsim u'$, the lists $L_\psi(u)$ and $L_\psi(u')$ are non-empty, and $|L_\psi(u)| + |L_\psi(u')| > \deg^{-\psi} v+k$, then
    we can extend $\psi$ to an $(L, \Pi)$-coloring $\psi'$ of $G[\dom(\psi)\cup\{u,u'\}]$ so that
    $\exc_{\psi'}(v)\ge k+1$.
\end{observation}
\begin{proof}
    If $\exc_\psi(v)>k$, then we can extend $\psi$ to $\psi'$ by choosing the colors for $u$ and $u'$ from their non-empty lists $L_\psi(u)$ and $L_\psi(u')$ arbitrarily.
    Hence, $\exc_\psi(v)=k$, and thus $\deg^{-\psi} v+k=|L_\psi(v)|$.

    If there exists a color $\alpha \in L_\psi(u)$ that blocks no color in $L_\psi(v)$, then choose a color $\beta\in L_\psi(u')$
    arbitrarily.  Similarly, if a color $\beta\in L_\psi(u')$ blocks no color in $L_\psi(v)$, then choose a color $\alpha \in L_\psi(u)$ arbitrarily.
    Finally, if every color in $L_\psi(u)$ blocks a color in $L_\psi(v)$ and every color in $L_\psi(u')$ blocks a color in $L_\psi(v)$,
    then since $|L_\psi(u)| + |L_\psi(u')| > \deg^{-\psi} v+k=|L_\psi(v)|$, we can choose colors $\alpha \in L_\psi(u)$ and $\beta \in L_\psi(u')$ that block the same color in $L_\psi(v)$.

    In either case, the color $\alpha$ at $u$ and the color $\beta$ at $u'$ together block at most one color in $L_\psi(v)$ at~$v$.
    Therefore, we can let $\psi'(u)=\alpha$ and $\psi'(u')=\beta$, and the conclusion follows from Observation~\ref{obs__improve}.
\end{proof}

Finally, let us remark that in correspondence coloring, the names of the colors play no intrinsic role,
and thus we can rename the colors as needed.  More precisely, for each vertex $v\in V(G)$, let $\sigma_v:\text{colors}\to\text{colors}$
be any bijection, and let $\sigma$ denote the system $\{\sigma_v:v\in V(G)\}$.  Let $L^\sigma$ be the list assignment giving to each vertex $v\in V(G)$ the list $L^\sigma(v)=\{\sigma_v(\alpha):\alpha\in L(v)\}$.
Similarly, let $\Pi^\sigma$ be the system of matchings assigning to each edge $e=uv\in E(G)$
the matching $\Pi^\sigma(e)=\{\{(u,\sigma_u(\alpha)),(v,\sigma_v(\beta))\}:\{(u,\alpha),(v,\beta)\}\in \Pi(e)\}$.
We say that $(L^\sigma,\Pi^\sigma)$ is the \emph{$\sigma$-renaming} of $(L,\Pi)$.
If $(L',\Pi')$ is the $\sigma$-renaming of $(L,\Pi)$ for some system $\sigma$, we say that $(L',\Pi')$ is \emph{obtained from $(L, \Pi)$ by renaming colors}.
Moreover, if $A$ is a subset of $V(G)$ such that $\sigma_u$ is the identity function for every vertex $u\in V(G)\setminus A$,
then we say that $(L',\Pi')$ is \emph{obtained from $(L, \Pi)$ by renaming colors on $A$}.
Note that if $\varphi$ is an $(L,\Pi)$-coloring of $G$, then $\sigma\circ\varphi$ is an $(L^\sigma,\Pi^\sigma)$-coloring of $G$.
Moreover, the renaming is reversible. That is, let $\sigma^{-1}$ denote the system $\{\sigma^{-1}_v:v\in V(G)\}$.
Then the $\sigma^{-1}$-renaming of a $\sigma$-renaming of $(L,\Pi)$ is equal to $(L,\Pi)$.
These observations imply that the following claim holds.

\begin{observation} \label{obs__renaming_at_vertex}
    Let $(L, \Pi)$ and $(L', \Pi')$ be correspondence assignments for $G$.
    If $(L', \Pi')$ is obtained from $(L, \Pi)$ by renaming colors, then $G$ is $(L, \Pi)$-colorable if and only if it is $(L', \Pi')$-colorable.
\end{observation}

For example, if all lists in $(L, \Pi)$ have the same size, then we can rename the colors so that all lists are identical.
More importantly, we can use the renaming to simplify our mental image of a correspondence assignment.
An edge $e = uv$ of $G$ is said to be \emph{straight} in a correspondence assignment $(L, \Pi)$ if
every pair $\{(u, \alpha), (v, \beta)\} \in \Pi(e)$ satisfies $\alpha=\beta$, i.e., if the conflicts
on $e$ form a subset of the conflicts for the usual proper coloring.
Of course, by renaming colors at the vertex $v$, we can ensure that the edge $e$ is straight in the resulting correspondence assignment.
This renaming can only affect the straightness of the rest of the edges incident with $v$.
Thus, if $T\subseteq G$ is a tree, then we can rename the colors at vertices of $T$ in order according to the distance
from a fixed root of $T$ to ensure that all edges of $T$ are straight
(in this case, we say that the subgraph $T$ is \emph{straight)}.  Thus, we get the following useful observation.

\begin{observation}\label{obs__straightening}
    Let $G$ be a graph, let $T \subseteq G$ be a tree and let $v_0$ be a vertex of $T$.  For every correspondence assignment $(L, \Pi)$ for $G$,
    there exists a correspondence assignment $(L', \Pi')$ obtained from $(L, \Pi)$ by renaming the colors on $V(T)\setminus\{v_0\}$
    such that $T$ is straight in $(L', \Pi')$.  Moreover, for any set $\Gamma\supseteq L(v_0)$ of size $\max \{|L(v)|:v\in V(T)\}$,
    we can assume that $L(v)\subseteq \Gamma$ for every vertex $v\in V(T)$.
\end{observation}

Finally, let us introduce one more property of exact correspondence assignments.  We say that an exact $k$-correspondence assignment $(L, \Pi)$ for a graph $G$ is \emph{full}
if for each edge $e=uv\in E(G)$, the matching $\Pi(e)$ is perfect (pairs each element of $\{u\}\times L(u)$ with exactly one element of $\{v\}\times L(v)$).  To show that
a graph is $k$-correspondence colorable, it clearly suffices to consider full exact $k$-correspondence assignments, since adding extra pairs to
the matchings on edges only makes it harder to color the graph.

\subsection{Probabilistic tools}

We use several standard results from probability theory, which can be found e.g.~in~\cite{alonspencer,MR01}.
In particular, we use the following lopsided version of Lovász Local Lemma.

\begin{theorem}[Lovász Local Lemma]\label{tool__lovasz_local_lemma}
    Consider a set $\EE$ of (bad) events and let $d$ be a positive integer.
    Suppose that for every event $A\in \EE$, there exists a set $\EE_A\subseteq \EE$ of size at most $d$
    such that for every $\EE'\subseteq \EE\setminus \EE_A$,
    $$\Pr{A\Bigm|\bigwedge_{A'\in \EE'} \lnot A'}\le \frac{1}{4d}.$$
    Then $\Pr{\bigwedge_{A\in \EE} \lnot A} > 0$.
\end{theorem}

\noindent We use the following Chernoff Bounds.

\begin{theorem}[Chernoff Bounds] \label{tool__chernoff}
    Let $X$ be a sum of independent Bernoulli variables.
    \begin{itemize}
    \item For any non-negative real numbers $\lambda$ and $t$ such that $\ex{X}\le \lambda$,
    $$\Pr{X\ge \lambda+t}\le \exp\Bigl(-\frac{t^2}{2\lambda+t}\Bigr).$$
    \item For any non-negative real numbers $\eta$ and $t$ such that $t\le \eta\le \ex{X}$,
    $$\Pr{X\le \eta-t}\le \exp\Bigl(-\frac{t^2}{2\eta}\Bigr).$$
    \end{itemize}
\end{theorem}

\noindent Also, at one spot, we need the following more general concentration result.

\begin{theorem}[McDiarmid's Inequality] \label{tool__mcdiarmid}
    Let $X$ be a random variable determined by $n$ independent trials $T_1, \ldots, T_n$,
    where for every $i\in[n]$, changing the outcome of $T_i$ can affect $X$ by at most $c_i$.
    For every non-negative real number $t$,
    $$\Pr{X<\ex{X}-t} \leq \exp\Bigl(-\frac{2t^2}{\sum_{i=1}^n c^2_i}\Bigr).$$
\end{theorem}

We also work with uniformly random correspondence colorings, where each of the possible $(L,\Pi)$-colorings of the considered graph $G$
is picked with the same probability.  The basic tool to argue about properties of such uniform colorings is \emph{resampling}.  Consider a function $f$ that
to each $(L,\Pi)$-coloring~$\psi$ of $G$ assigns a subset $f(\psi)$ of $V(G)$; we call such a function a \emph{subset selector}.
For a set $S\subseteq V(G)$, an $(L,\Pi)$-coloring $\psi'$ is an \emph{$S$-modification} of $\psi$ if $\psi'(v)=\psi(v)$ holds for every $v\in V(G)\setminus S$.
An $(L,\Pi)$-coloring $\psi''$ is an \emph{$f$-invariant modification}
of $\psi$ if $\psi''$ is an $f(\psi)$-modification of $\psi$ such that $f(\psi'')=f(\psi)$;
that is, we change the coloring only on the vertices of the set $f(\psi)$, but we are careful not to change the value of $f$.
An \emph{$f$-invariant resampling} of $\psi$ is an $f$-invariant modification of $\psi$ chosen uniformly at random,
i.e., each of the $f$-invariant modifications of $\psi$ is picked with the same probability.

By the following observation, $f$-invariant resampling preserves the uniformity of a random $(L,\Pi)$-coloring;
thus, to show that a uniform $(L,\Pi)$-coloring $\psi$ of $G$ has some property with high probability, we can
instead argue that an $f$-invariant resampling of $\psi$ does.
This observation can also be seen as a consequence of the so-called {\em spatial Markov property} for the hard-core model, see e.g.~\cite{DaKa25+}.

\begin{lemma}\label{tool__resampling}
Let $G$ be a graph, let $(L,\Pi)$ be a correspondence assignment for $G$, and let $f$ be a subset selector.
Let $\psi$ be a uniformly random $(L,\Pi)$-coloring of $G$, and let $\psi'$ be an $f$-invariant resampling of $\psi$.
Then $\psi'$ is also a uniformly random $(L,\Pi)$-coloring of $G$.
\end{lemma}
\begin{proof}
Let $n$ be the number of $(L,\Pi)$-colorings of $G$, and for each set $S\subseteq V(G)$ and every $(L,\Pi)$-coloring $\sigma$ of $G-S$,
let $n_{S,\sigma}$ be the number of $(L,\Pi)$-colorings $\varphi$ of $G$ such that $f(\varphi)=S$ and $\varphi$ extends $\sigma$.

Consider a fixed $(L,\Pi)$-coloring $\varphi$ of $G$, let $S=f(\varphi)$, and let $\sigma$ be the restriction of $\varphi$ to $V(G)\setminus S$.
For the $f$-invariant resampling $\psi'$ of $\psi$ to be equal to $\varphi$, we must have
\begin{itemize}
\item $f(\psi)=S$ and $\psi$ must extend $\sigma$, 
\end{itemize}
which happens with probability $\frac{n_{S,\sigma}}{n}$, and furthermore
\begin{itemize}
\item $\psi'$ must be chosen to be $\varphi$ from all $n_{S,\sigma}$ colorings with these properties, 
\end{itemize}
which happens with probability~$\tfrac{1}{n_{S,\sigma}}$. Thus,
$$\Pr{\psi'=\varphi}=\frac{n_{S,\sigma}}{n}\cdot\frac{1}{n_{S,\sigma}}=\frac{1}{n}.$$
Since this holds for every $(L,\Pi)$-coloring $\varphi$ of $G$, we conclude that $\psi'$ is a uniformly random $(L,\Pi)$-coloring of $G$.
\end{proof}

% Forbidden induced subgraphs
\newpage

 \section{Forbidden induced subgraphs} \label{section_03}

For a positive integer $\Delta$, a \emph{$\Delta$-counterexample} is a triple $(G,L,\Pi)$,
where $G$ is a graph of maximum degree at most $\Delta$ and clique number at most $\Delta-1$
and $(L,\Pi)$ is a full exact $(\Delta-1)$-correspondence assignment for $G$ such that $G$ is not $(L,\Pi)$-colorable.
We aim to prove Theorem~\ref{thm__bk_conjecture_is_true_asym_for_DP_coloring} by showing that
no $\Delta$-counterexamples exist for any sufficiently large $\Delta$.
For contradiction, we suppose that a $\Delta$-counterexample exists, and we consider a \emph{minimal} one,
i.e., one with $|V(G)|$ smallest possible.  In this section, we restrict the structure of minimal $\Delta$-counterexamples
by finding a number of \emph{reducible configurations}, induced subgraphs that cannot appear in a minimal $\Delta$-counterexample,
and using them to argue about the neighborhoods of large cliques.
Let us start with two simple observations.

\begin{observation}\label{obs__mindeg}
    For any positive integer $\Delta$, if $(G,L,\Pi)$ is a minimal $\Delta$-counterexample, then $\delta(G)\ge \Delta-1$.
\end{observation}
\begin{proof}
If $G$ contained a vertex of degree at most $\Delta-2$, then note that $G-v$ would have an $(L,\Pi)$-coloring by the minimality
of the $\Delta$-counterexample, and that this coloring could be extended greedily to an $(L,\Pi)$-coloring of $G$,
since $|L(v)|\ge \Delta-1>\deg v$.
\end{proof}

\begin{observation} \label{obs__forbidden_d1_corr_col_ind_subgraph}
    For any positive integer $\Delta$, if $(G,L,\Pi)$ is a minimal $\Delta$-counterexample, then no induced subgraph $H$ of $G$ is $\minusjedna_H$-correspondence colorable.
\end{observation}
\begin{proof}
Since $(G,L,\Pi)$ is a minimal $\Delta$-counterexample, there exists an $(L,\Pi)$-coloring $\psi$ of $G-V(H)$.
Since $\psi$ does not extend to an $(L,\Pi)$-coloring of $G$, Observation~\ref{obs__combine} implies that $H$ is not $(L_\psi,\Pi_\psi)$-colorable.
Since $(L_\psi,\Pi_\psi)$ is a $\minusjedna_H$-correspondence assignment by Observation~\ref{obs__makesminusjedna},
the induced subgraph $H$ is not $\minusjedna_H$-correspondence colorable.
\end{proof}

Given Observation~\ref{obs__forbidden_d1_corr_col_ind_subgraph}, our focus now will be on showing that specific graphs are $\minusjedna$-correspondence
colorable.  Let us remark that, with the exception of Lemma~\ref{lemma__no_bad_vertex_in_D-1_clique}, the analogous results are known to be true in the list coloring setting~\cite{CKR23},
and our proofs of these results are essentially just reformulations to the correspondence coloring setting.

\subsection{Graphs with many universal vertices}

We first aim to show that for any graph $H$ with many universal vertices, if $\omega(H)\le |V(H)|-2$, then $H$ is $\minusjedna$-correspondence
colorable.  Let us start by proving several special cases of this claim.

\begin{lemma} \label{lemma__K3_d1_corr_col}
    For every integer $t \geq 6$, the graph $G=K_t \vee \overline{K_3}$ is $\minusjedna$-correspondence colorable.
\end{lemma}

\begin{proof}
    Suppose that $V(K_t) = \{x_1, x_2, \ldots, x_t\}$ and $V(\overline{K_3}) = \{a, b, c\}$. Let $(L, \Pi)$ be an exact $\minusjedna$-correspondence assignment for $G$;
    we have $|L(x_i)|=t+1$ for $i\in [t]$ and $|L(a)|=|L(b)|=|L(c)|=t-1$.
    
    Since $|L(x_1)|=|L(x_2)|=|L(a)|+2$, for each $i\in [2]$, there exist distinct colors $\alpha_i,\alpha'_i\in L(x_i)$ which block no colors at $a$.
    Since $\alpha_1$ at $x_1$ blocks at most one color at $x_2$, we can assume that it does not block $\alpha_2$ at $x_2$.
    Let $\psi_1$ be the partial $(L, \Pi)$-coloring of $G$ giving $x_1$ the color $\alpha_1$ and $x_2$ the color $\alpha_2$.
    The choice of $\alpha_1$ and $\alpha_2$ ensures that $a$ has $\psi_1$-excess $1$.

    Note that $\exc_{\psi_1}(x_3)\ge -1$, $\deg^{-\psi_1} x_3=t$, and since $t\ge 6$, we have
    $$|L_{\psi_1}(b)| + |L_{\psi_1}(c)|\ge (t-3)+(t-3)>t-1.$$
    By Observation~\ref{obs__blocking_at_most_one_color}, we can extend $\psi_1$ to a partial $(L, \Pi)$-coloring $\psi$ by choosing colors for $b$ and $c$ so that
    $x_3$ has non-negative $\psi$-excess.  Moreover, we have $\exc_{\psi}(a)\ge \exc_{\psi_1}(a)=1$.
    Observation~\ref{obs__greedy} with the ordering $x_4$, \ldots, $x_t$, $x_3$, $a$
    shows that $\psi$ extends to an $(L, \Pi)$-coloring of $G$.

    Therefore, the graph $G$ is $\minusjedna$-correspondence colorable.
\end{proof}

Next, we show another lemma with essentially the same proof.

\begin{lemma} \label{lemma__P4_2K2_d1_corr_col}
    For every integer $t \geq 5$ and each $H \in \{P_4, 2K_2\}$, the graph $G=K_t \vee H$ is $\minusjedna$-correspondence colorable.
\end{lemma}
\begin{proof}
    Suppose that $V(K_t) = \{x_1, x_2, \ldots, x_t\}$ and $V(H) = \{a, b, c, d\}$, where $\deg_H a = 1$ and $b \nsim d$.
    Let $(L, \Pi)$ be an exact $\minusjedna$-correspondence assignment for $G$;
    we have $|L(x_i)|=t+2$ for $i\in [t]$, $|L(y)|\ge t-1$ for $y\in V(H)$, and $|L(a)|=t$.

    Since $|L(x_1)|=|L(x_2)|=|L(a)|+2$, we can choose colors $\alpha_1\in L(x_1)$ and $\alpha_2\in L(x_2)$ so that $\alpha_1$ at $x_1$ and $\alpha_2$ at $x_2$
    block no colors at $a$ and $\{(x_1,\alpha_1),(x_2,\alpha_2)\}\not\in \Pi(x_1x_2)$.
    Let $\psi_1$ be the partial $(L, \Pi)$-coloring of $G$ giving $x_1$ the color $\alpha_1$ and $x_2$ the color $\alpha_2$.
    The choice of $\alpha_1$ and $\alpha_2$ ensures that $a$ has $\psi_1$-excess~$1$.

    Note that $\exc_{\psi_1}(x_3)\ge -1$, $\deg^{-\psi_1} x_3=t+1$, and
    $$|L_{\psi_1}(b)| + |L_{\psi_1}(d)| \ge (t-2) + (t-2)>t,$$
    since $t\ge 5$.
    By Observation~\ref{obs__blocking_at_most_one_color}, we can extend $\psi_1$ to a partial $(L, \Pi)$-coloring $\psi$ by choosing colors for $b$ and $d$ 
    so that $x_3$ has non-negative $\psi$-excess.  Moreover, we have $\exc_{\psi}(a)\ge \exc_{\psi_1}(a)=1$.
    Observation~\ref{obs__greedy} with the ordering $c$, $x_4$, \ldots, $x_t$, $x_3$, $a$
    shows that $\psi$ extends to an $(L, \Pi)$-coloring of $G$.

    Therefore, the graph $G$ is $\minusjedna$-correspondence colorable.
\end{proof}

Next, we want to consider the more complicated case of $K_t \vee C_4$.  Let us start with an auxiliary lemma.
\begin{lemma} \label{lemma__C4_d1_corr_col__claim3}
    The graph $G=K_1 \vee C_4$ is 3-correspondence colorable.
\end{lemma}
    \begin{proof}
        Suppose that $V(K_1) = \{x\}$ and $C_4 = abcd$. Let $(L, \Pi)$ be an exact $3$-correspondence assignment for $G = K_1 \vee C_4$.
	By Observation~\ref{obs__straightening}, we can assume that $L(a)=L(b)=L(c)=L(d)=[3]$ and the path $abcd$ is straight.
	Since the color $1$ at $a$ blocks only one color at $d$, we can (by renaming the colors $2$ and $3$ on $\{a,b,c,d\}$ if needed) assume that $\{(a,1),(d,2)\}\not\in \Pi(ad)$.

	Since $|L(x)|=3$, there exists a color $\alpha\in L(x)$ that does not block colors $1$ and $2$ at $a$. Let $\psi$ be the partial $(L,\Pi)$-coloring of $G$ giving $x$ the color $\alpha$.
	If there exists a vertex $u\in\{a,b,c,d\}$ such that $|L_\psi(u)|=3$, then note that $u$ has $\psi$-excess $1$ and all other vertices of $C_4$ have non-negative $\psi$-excess,
	and thus we can extend $\psi$ to an $(L, \Pi)$-coloring of $G$ using Observation~\ref{obs__greedy}.  Hence, we can assume $|L_\psi(u)|=2$ for each $u\in \{a,b,c,d\}$,
	and in particular $L_{\psi}(a)=[2]$ by the choice of $\alpha$.

	Similarly, if there exists an edge $uv\in E(C_4)$ such that a color $\beta\in L_{\psi}(u)$ does not block any color from $L_{\psi}(v)$, then we can color $u$ by $\beta$
	and further extend the coloring to an $(L, \Pi)$-coloring of $G$ using Observation~\ref{obs__greedy}.

	If that is not the case, then since the edges $ab$, $bc$, and $cd$ are straight and $L_{\psi}(a)=[2]$, we also have $L_{\psi}(b)=L_{\psi}(c)=L_{\psi}(d)=[2]$.
	Since $\{(a,1),(d,2)\}\not\in \Pi(ad)$, we can extend $\psi$ to an $(L, \Pi)$-coloring of $G$ by giving $a$ and $c$ the color $1$ and $b$ and $d$ the color $2$.
	We conclude that the graph $G$ is 3-correspondence colorable.
    \end{proof}  

\begin{lemma} \label{lemma__C4_d1_corr_col}
    For every integer $t \geq 5$, the graph $G=K_t \vee C_4$ is $\minusjedna$-correspondence colorable.
\end{lemma}
\begin{proof}
    Suppose that $V(K_t) = \{x_1, x_2, \ldots, x_t\}$ and $C_4 = y_1y_2y_3y_4$.
    Let $(L, \Pi)$ be an exact $\minusjedna$-correspondence assignment for $G$;
    by Observation~\ref{obs__straightening}, we can assume that this correspondence assignment is straight on the star $S$ with center $x_1$ and rays $x_2$, \ldots, $x_t$.
    We have $|L(x_i)|=t+2$ for $i\in[t]$ and $|L(y_j)|=t+1$ for $j\in[4]$,
    and thus for every $i\in[t]$ and $j\in[4]$, there exists a color $\alpha_{i,j} \in L(x_i)$ that blocks no color at $y_j$.
    Let us distinguish two cases.

\begin{itemize}
\item
    Suppose first that there exist $j\in[4]$ and distinct $i_1,i_2\in [t]$ such that $\alpha_{i_1,j}$ at $x_{i_1}$ does not block $\alpha_{i_2,j}$ at $x_{i_2}$.
    By symmetry, we can assume $i_1=1$, $i_2=2$, and $j=1$.
    Let $\psi_1$ be the partial $(L, \Pi)$-coloring of $G$ giving $x_1$ the color $\alpha_{1,1}$ and $x_2$ the color $\alpha_{2,1}$.
    Thus, $y_1$ has $\psi_1$-excess~$1$.
    Note that $\exc_{\psi_1}(x_3)\ge -1$, $\deg^{-\psi_1} x_3=t+1$, and
    $$|L_{\psi_1}(y_2)| + |L_{\psi_1}(y_4)| \ge (t-1) + (t-1)>t,$$
    since $t\ge 5$.
    By Observation~\ref{obs__blocking_at_most_one_color}, we can extend $\psi_1$ to a partial $(L, \Pi)$-coloring $\psi$ by choosing colors for $y_2$ and $y_4$
    so that $x_3$ has non-negative $\psi$-excess.  Moreover, we have $\exc_{\psi}(y_1)\ge \exc_{\psi_1}(y_1)=1$.
    Observation~\ref{obs__greedy} with the ordering $y_3$, $x_4$, \ldots, $x_t$, $x_3$, $y_1$
    shows that $\psi$ extends to an $(L, \Pi)$-coloring of $G$.

\item
    Hence, we can assume that for every $j\in[4]$ and all distinct $i_1,i_2\in [t]$, the color $\alpha_{i_1,j}$ at $x_i$ blocks the color $\alpha_{i_2,j}$ at $x_j$.
    Since $(L, \Pi)$ is straight on the star $S$, this implies that $\alpha_{i,j}=\alpha_{1,j}$ for every $i\in[t]$ and $j\in[4]$.
    For each $j\in[4]$, let $\alpha_j=\alpha_{1,j}$.  Thus, for each $i\in[t]$, the color $\alpha_j$ appears in $L(x_i)$ and at $x_i$ blocks no color at $y_j$,
    and for distinct $i_1,i_2\in [t]$, the color $\alpha_j$ at $x_{i_1}$ blocks the color $\alpha_j$ at $x_{i_2}$.

    Let $A=\{\alpha_1,\ldots,\alpha_4\}$ and $a=|A|$.  Note that the colors $\alpha_1$, \ldots, $\alpha_4$ do not have to be all distinct, and thus $a<4$ is possible.
    Let $\psi$ be a partial $(L,\Pi)$-coloring of $G$ that assigns to each of the vertices $x_1$, \ldots, $x_a$ a distinct color from $A$,
    and for $i=a+1,\ldots,t-1$ assigns to $x_i$ a color from $L(x_i)$ not blocked by the colors assigned to the preceding vertices
    (this is possible since $|L(x_i)|=t+2>i-1$).

    Let $H=G-\dom(\psi)$, and note that $H=x_t\vee C_4$.
    The usage of all colors from $A$ ensures that the $\psi$-excess of all vertices of $C_4$ is non-negative, and thus $|L_\psi(y_j)|\ge 3$
    for every $j\in[4]$.  Moreover, since $(L,\Pi)$ is a $\minusjedna$-correspondence assignment, the $\psi$-excess of $x_t$ is at least $-1$,
    and thus $|L_\psi(x_t)|\ge 3$.  By Lemma~\ref{lemma__C4_d1_corr_col__claim3}, $H$ has an $(L_\psi,\Pi_\psi)$-coloring,
    which combines with $\psi$ for an $(L, \Pi)$-coloring of $G$.
    \end{itemize}

    Therefore, the graph $G$ is $\minusjedna$-correspondence colorable.
\end{proof}

\noindent We now need a simple observation based on the presence of the induced subgraphs considered in the preceding lemmas.

\begin{observation}\label{obs__omegami}
If a graph $G$ does not contain any induced subgraph isomorphic to $\overline{K_3}, 2K_2, P_4$, or $C_4$, then
$\omega(G)\ge |V(G)|-1$.
\end{observation}
\begin{proof}
Since $P_4\not\le G$, it follows that $G$ is a \emph{cograph}, i.e., $G$ is obtained by iterated complete joins and disjoint unions from single-vertex graphs.

We prove the claim by induction on $|V(G)|$.  If $|V(G)|=1$, then $\omega(G)=|V(G)|$.  Otherwise, $G$ is the disjoint union or complete join of smaller
graphs $G_1$ and $G_2$.
\begin{itemize}
\item If $G$ is the disjoint union of $G_1$ and $G_2$, then since $\overline{K_3}\not\le G$, both $G_1$ and $G_2$ must be cliques.
Moreover, since $2K_2\not\le G$, one of $G_1$ and $G_2$ must be a single vertex.  Thus, $\omega(G)=\max(|V(G_1)|,|V(G_2)|)=|V(G)|-1$.
\item If $G$ is the complete join of $G_1$ and $G_2$, then since $C_4\not\le G$, we can by symmetry assume that $G_2$ is a clique.
Note that $G_1$ is an induced subgraph of $G$, and thus $\overline{K_3}, 2K_2, P_4, C_4\not\le G_1$.
By the induction hypothesis, we have $\omega(G_1)\ge |V(G_1)|-1$, and thus
$\omega(G)=\omega(G_1)+|V(G_2)|\ge |V(G)|-1$. \qedhere
\end{itemize}
\end{proof}

\noindent Thus, we can now give the following common generalization of Lemmas~\ref{lemma__K3_d1_corr_col}, \ref{lemma__P4_2K2_d1_corr_col}, and \ref{lemma__C4_d1_corr_col}.

\begin{lemma} \label{lemma__d1_corr_col}
    Let $G$ be a graph with $\omega(G) \leq |V(G)|-2$.
    If $G$ has at least $6$ universal vertices, then it is $\minusjedna$-correspondence colorable.
\end{lemma}
\begin{proof}
    Let $(L, \Pi)$ be a $\minusjedna$-correspondence assignment for $G$.
    Let $U$ be the set of all universal vertices of $G$; note that these vertices induce a clique $K_U$.  Let $t=|U|\ge 6$.
    By Observation~\ref{obs__omegami}, there exists a set $S\subseteq V(G)$ such that the induced subgraph $H=G[S]$ is isomorphic to $\overline{K_3}$, $2K_2$, $P_4$, or $C_4$.
    Since each vertex of $S$ is incident with a non-edge of $H$, we have $S\cap U=\emptyset$.

    Let $G_1=G-(U\cup S)$.  Note that for each vertex $v\in V(G_1)$, we have
    $|L(v)|\ge \deg_G v-1\ge (\deg_{G_1} v+t)-1>\deg_{G_1} v$.
    Therefore, the graph $G_1$ has an $(L, \Pi)$-coloring $\psi$, obtained by coloring its vertices greedily.
    Let $G_2=G-\dom(\psi)=G[U\cup S]=K_U\vee H$.
    Since $(L, \Pi)$ is a $\minusjedna_G$-correspondence assignment, $(L_\psi,\Pi_\psi)$ is a $\minusjedna_{G_2}$-correspondence assignment.
    By Lemmas~\ref{lemma__K3_d1_corr_col}, \ref{lemma__P4_2K2_d1_corr_col}, and \ref{lemma__C4_d1_corr_col},
    it follows that $G_2$ has an $(L_\psi,\Pi_\psi)$-coloring, which combines with $\psi$ for an $(L, \Pi)$-coloring of $G$.

    Therefore, the graph $G$ is $\minusjedna$-correspondence colorable.
\end{proof}

Lemma~\ref{lemma__d1_corr_col} has the following consequence, bounding the density
of non-$\minusjedna$-correspondence colorable graphs with large clique number.

\begin{lemma}\label{lemma__densecom}
Let $a\ge 2$ be an integer and let $G$ be a graph with $\omega(G)=|V(G)|-a$.
If $\delta(G)\ge \tfrac{1}{2}(|V(G)|+a+4)$, then the graph $G$ is $\minusjedna$-correspondence colorable.
\end{lemma}
\begin{proof}
    Let $(L, \Pi)$ be a $\minusjedna$-correspondence assignment for $G$.
    Let $n=|V(G)|$ and let $K$ be a clique of size $\omega(G)=n-a$ in $G$; note that $n=\omega(G)+a\ge a+1$.
    Since $a\ge 2$, there exist distinct vertices $u,v\in V(G)\setminus V(K)$.  Let $R=V(G)\setminus (V(K)\cup \{u,v\})$
    and note that for each vertex $x\in V(G)$, we have
    $$|L(x)|\ge \delta(G)-1\ge \tfrac{1}{2}(n+a+2)\ge \tfrac{1}{2}(2a+3)>a-2=|R|.$$
    Hence, the graph $G[R]$ has an $(L, \Pi)$-coloring $\psi$, obtained by coloring its vertices greedily.

    Let $G'=G-R$ and note that $(L_\psi,\Pi_\psi)$ is a $\minusjedna_{G'}$-correspondence assignment for $G'$.
    Since $G'\le G$ and $K$ is a clique in $G'$, we have $\omega(G')=n-a=|V(G')|-2$.
    Note that the number of neighbors of each of $u$ and $v$ in $K$ is at least
    $$\delta(G)-(|R|+1)\ge \tfrac{1}{2}(n+a+4)-(a-1)=\tfrac{1}{2}(n-a+6)=\tfrac{1}{2}(|V(K)|+6).$$
    Therefore, the vertices $u$ and $v$ have at least $2\cdot \tfrac{1}{2}(|V(K)|+6)-|V(K)|=6$
    common neighbors in $K$.  Since $K$ is a clique, these common neighbors are universal vertices of $G'$.
    Therefore, by Lemma~\ref{lemma__d1_corr_col}, the graph $G'$ has an $(L_\psi,\Pi_\psi)$-coloring, which combines with $\psi$
    for an $(L, \Pi)$-coloring of $G$.

    Therefore, the graph $G$ is $\minusjedna$-correspondence colorable.
\end{proof}

\subsection{Vertex neighborhoods}

Let us now consider graphs with at least one universal vertex.
These will be quite useful in the proof of Theorem~\ref{thm__bk_conjecture_is_true_asym_for_DP_coloring},
as using them we can restrict subgraphs induced by closed neighborhoods of vertices in a minimal $\Delta$-counterexample.
Let us start by considering the situation when this neighborhood contains three disjoint non-edges.

\begin{lemma} \label{lemma__near_clique_three_non-edges}
    Let $B$ be a graph of minimum degree at least $\tfrac{1}{2}(|V(B)|+5)$.
    If $\overline{B}$ contains a matching of size three, then the graph $B \vee v$ is $\minusjedna$-correspondence colorable.
\end{lemma}
\begin{proof}
    Suppose that $x$, $x'$, $y$, $y'$, $z$, and $z'$ are distinct vertices in $B$ such that $x \nsim x'$, $y \nsim y'$, and $z \nsim z'$.
    Let $n=|V(B)|$.
    Note that any distinct non-adjacent vertices $a$ and $b$ of $B$ have at least
    $$\deg_B a+\deg_B b-(n-2)\ge 2\delta(B)-n+2\ge 7$$
    common neighbors in $B$, and thus the graph $B'=B-\{x,x',y,y',z,z'\}$ is connected.
    Moreover, by considering $a=x$ and $b=x'$, we conclude that the vertices $x$ and $x'$ have a common neighbor $w$ in $B$ different from $y$, $y'$, $z$, and $z'$.

    Let $(L, \Pi)$ be an exact $\minusjedna$-correspondence assignment for the graph $G=B \vee v$.
    Note that $\deg_G w\le n$ and
    $$|L(x)| + |L(x')| \geq 2\delta(B) > n-1.$$
    Hence, by Observation~\ref{obs__blocking_at_most_one_color} (with the empty precoloring) we can choose colors for $x$ and $x'$, obtaining
    a partial $(L, \Pi)$-coloring $\psi_1$ such that $w$ has a non-negative $\psi_1$-excess.

    Note that $\exc_{\psi_1}(v)\ge -1$, $\deg^{-\psi_1} v=n-2$, and
    $$|L_{\psi_1}(y)| + |L_{\psi_1}(y')| \geq 2(\delta(B)-2) > n - 3.$$
    By Observation~\ref{obs__blocking_at_most_one_color}, we can extend $\psi_1$ to a partial $(L, \Pi)$-coloring $\psi_2$ by choosing colors for $y$ and $y'$
    so that that $v$ has a non-negative $\psi_2$-excess.

    We now repeat the same argument for $z$, $z'$, and $v$:
    Since $\exc_{\psi_2}(v)\ge 0$, $\deg^{-\psi_2} v=n-4$, and
    $$|L_{\psi_2}(z)| + |L_{\psi_2}(z')| \geq 2(\delta(B)-4) > n - 4.$$
    Observation~\ref{obs__blocking_at_most_one_color} implies that we can extend $\psi_2$ to a partial $(L, \Pi)$-coloring $\psi$ by choosing colors for $z$ and $z'$ 
    so that $v$ has a positive $\psi$-excess.

    Moreover, we have $\exc_\psi(w)\ge\exc_{\psi_1}(w)\ge 0$.
    Therefore, we can extend $\psi$ to $B \vee v$ greedily using Observation~\ref{obs__greedy};
    we order the vertices of the connected subgraph $B'$ in non-increasing order according to the distance from $w$ and put the vertex $v$
    at the end to ensure that the assumptions are satisfied.
\end{proof}

By combining Lemmas~\ref{lemma__densecom} and \ref{lemma__near_clique_three_non-edges}, we obtain the following important consequence.

\begin{lemma} \label{lemma__near_clique}
    Let $B$ be a graph with $\delta(B) \geq \tfrac{1}{2}(|V(B)|+7)$.
    If $\omega(B) \leq |V(B)| - 2$, then the graph $G=B \vee v$ is $\minusjedna$-correspondence colorable.
\end{lemma}
\begin{proof}
    Let $\mu$ be the maximum size of a matching in $\overline{B}$.
    If $\mu\ge 3$, then the conclusion follows by Lemma~\ref{lemma__near_clique_three_non-edges},
    and thus assume that $\mu\le 2$.  Note that the set of vertices not incident with the maximum matching
    forms an independent set in $\overline{B}$, and thus $\omega(B)=\alpha(\overline{B})\ge |V(B)|-2\mu\ge |V(B)|-4$.
    Since $G$ is obtained from $B$ by adding a universal vertex, it follows that $\omega(G)=|V(G)|-a$ for an integer $a$
    such that $2\le a\le 4$.  Moreover, $\delta(G)=\delta(B)+1\ge \tfrac{1}{2}(|V(B)|+9)\ge \tfrac{1}{2}(|V(G)|+8)\ge \tfrac{1}{2}(|V(G)|+a+4)$.
    Therefore, the graph $G$ is $\minusjedna$-correspondence colorable by Lemma~\ref{lemma__densecom}.
\end{proof}

This leads us to the first key result of this section, showing that in a minimal $\Delta$-counterexample, the neighborhood of each vertex contains either a large
clique or many non-edges.
\begin{theorem} \label{thm__many_nones}
    Let $H$ be a graph with $n$ vertices, where $H$ is not a clique, and let $v$ be a vertex not belonging to $H$.
    Suppose that for every induced subgraph $B$ of $H$, the graph $B\vee v$ is not $\minusjedna$-correspondence colorable.
    Then $H$ has at least $$\frac{(n-\omega(H)-1)(n+\omega(H)-14)}{4}$$ non-edges.
\end{theorem}
\begin{proof}
Since $H$ is not a clique, we have $\omega(H)\le n-1$.
Let $H_1=H$, and for $i=1,\ldots, n$, let $v_i$ be a vertex of minimum degree in $H_i$ and let $H_{i+1}=H_i-v_i$.
Consider any positive integer $i\le n-\omega(H)-1$.  The graph $H_i$ has $n-i+1\ge \omega(H)+2\ge\omega(H_i)+2$ vertices.
Since the graph $H_i\vee v$ is not $\minusjedna$-correspondence colorable, Lemma~\ref{lemma__near_clique} implies that
$\deg_{H_i} v_i=\delta(H_i)\le \tfrac{1}{2}(|V(H_i)|+6)$.  Consequently, the number of non-edges of $H_i$ incident with $v_i$ is
$$|V(H_i)|-\deg_{H_i} v_i-1\ge |V(H_i)|-\frac{|V(H_i)|+6}{2}-1=\frac{|V(H_i)|-8}{2}=\frac{n-i-7}{2}.$$
Therefore, the total number of non-edges of $H$ is at least
\[\sum_{i=1}^{n-\omega(H)-1} \frac{n-i-7}{2}=\frac{(n-\omega(H)-1)(n+\omega(H)-14)}{4}.\qedhere
\]
\end{proof}

\subsection{Near-clique neighborhoods}

Next, we study neighborhoods of large cliques (or more precisely, \emph{near-cliques}, see below) in minimal $\Delta$-counterexamples.
Let us start by excluding another configuration.

\begin{lemma}\label{lemma__33vamp}
    Let $H$ be a graph containing distinct vertices $u_1$ and $u_2$ such that $H-\{u_1,u_2\}$ is
    a clique.  Let $C$ denote this clique and for $i\in [2]$, let $U_i$ be the set of vertices of $C$ adjacent to $u_i$ and non-adjacent to $u_{3-i}$.
    Suppose that $|U_1|\ge 3$, $|U_2|\ge 3$, and $u_1$ has at least five neighbors in $C$.
    If $u_1$ and $u_2$ have at least one common neighbor $v\in V(C)$, then the graph $H$ is $\minusjedna$-correspondence colorable.
\end{lemma}
\begin{proof}
    Let $n=|V(H)|$.
    Note that for each $x\in U_1$, we have $\deg x=n-2\ge |V(H)\setminus (U_2\cup \{u_1\})|+2\ge \deg u_1+2$.
    Let $x_1$, $x_2$, and $x_3$ be distinct vertices in $U_1$.  Let $y\in V(C)\setminus\{v,u_1,u_2,u_3\}$ be a neighbor of $u_1$,
    which exists since $u_1$ has at least five neighbors in $C$.

    Let $(L, \Pi)$ be an exact $\minusjedna$-correspondence assignment for the graph $H$.
    Since $|L(x_1)|\ge |L(u_1)|+2$ and $|L(x_2)|\ge |L(u_1)|+2$, we can choose colors for $x_1$ and $x_2$ and obtain
    a partial $(L,\Pi)$-coloring $\psi_1$ of $H$ such that $\psi_1(x_1)$ and $\psi_1(x_2)$ do not block any colors at $u_1$.
    Consequently, $u_1$ has $\psi_1$-excess $1$.

    Note that $\exc_{\psi_1}(v)\ge -1$, $\deg^{-\psi_1} v=n-3$,
    $$|L_{\psi_1}(u_2)|=|L(u_2)|=\deg u_2-1\ge |U_2\cup\{v\}|-1\ge 3,$$
    and
    $$|L_{\psi_1}(x_3)| + |L_{\psi_1}(u_2)| \ge (n-5)+3 > n - 4.$$
    By Observation~\ref{obs__blocking_at_most_one_color}, we can extend $\psi_1$ to a partial $(L, \Pi)$-coloring $\psi$ by choosing colors for $x_3$ and $u_2$
    so that that $v$ has a non-negative $\psi$-excess.

    We can now extend $\psi$ to $H$ greedily using Observation~\ref{obs__greedy}, with the vertices $y$, $v$, and $u_1$ colored last.
\end{proof}

A set $D$ of vertices of a graph $G$ is called a \emph{near-clique} if $\omega(G[D]) \geq |D| - 1$.
A subset $C\subseteq D$ of size $\omega(G[D])$ inducing a clique in $G$ is a \emph{core} of the near-clique (the core is unique unless $G[D]$ has exactly one non-edge).
For a positive integer $\Delta$, the set $D$ is \emph{$\Delta$-isolated} in $G$ if every vertex $v\in V(G)\setminus D$ has at most $\lceil \tfrac{3}{4} \Delta \rceil$ neighbors in $D$.
In Section~\ref{section_04}, we show that each minimal $\Delta$-counterexample admits a sparse-dense decomposition, where the dense parts of this decomposition
are formed by $\Delta$-isolated near-cliques.  Thus, we need to understand how such near-cliques interact with the rest of the graph.
We often use the following simple observation.
\begin{observation}\label{obs__nns}
Let $\Delta$ be a positive integer, let $D$ be a $\Delta$-isolated near-clique in a graph $G$, let $C$ be a core of $D$, and let $p=\Delta-|C|$.
Then the vertex in $D\setminus C$ (if any) has at least one non-neighbor in $C$, and every vertex in $V(G)\setminus D$ has
at least $\tfrac{\Delta-3}{4}-p$ non-neighbors in $C$.
\end{observation}

For a positive integer $m$, a vertex $u\in V(G)\setminus C$ is \emph{$(C,m)$-loose} if $u$ has at most $m$ neighbors in $C$, and \emph{$(C,m)$-tight} otherwise.

\begin{lemma} \label{lemma__no_two_vertices_can_see_together_many_vertices_in_max_clique}
    Let $\Delta$ be a positive integer and let $(G,L,\Pi)$ be a minimal $\Delta$-counterexample.
    Let $D$ be a $\Delta$-isolated near-clique in $G$, let $C$ be a core of $D$, and let $p=\Delta-|C|$.
    If $\Delta\ge 4p+8$, then every vertex $v\in C$ has at most one $(C,7)$-tight neighbor.
\end{lemma}
\begin{proof}
    Suppose for a contradiction that $v$ has (at least) two distinct $(C,7)$-tight neighbors $u_1$ and $u_2$.
    By Observation~\ref{obs__nns}, $u_1$ and $u_2$ each have a non-neighbor in $C$,
    and since $\Delta\ge 4p+8$, at least one of them (one not belonging to $D$) has more than one non-neighbor in $C$.
    Thus, we can select distinct vertices $u'_1, u'_2\in C$ such that $u_1\nsim u'_1$ and $u_2\nsim u'_2$.
    
    Let $M$ be the set of common neighbors of $u_1$ and $u_2$ in $C$ (including the vertex $v$),
    and consider the subgraph $H$ of $G$ induced by $M\cup \{u_1,u'_1,u_2,u'_2\}$.  Since $u_1\nsim u'_1$ and $u_2\nsim u'_2$,
    we have $\omega(H)\le |V(H)|-2$.  By Observation~\ref{obs__forbidden_d1_corr_col_ind_subgraph}, the induced subgraph $H$ is not $\minusjedna_H$-correspondence colorable,
    and since the vertices of $M$ are universal in $H$, Lemma~\ref{lemma__d1_corr_col} implies $|M|\le 5$.

    For $i\in [2]$, let $U_i$ be the set of vertices of $C$ adjacent to $u_i$ and non-adjacent to $u_{3-i}$.
    Since the vertices $u_1$ and $u_2$ are $(C,7)$-tight and $|M|\le 5$, we have $|U_1|, |U_2|\ge 3$.
    However, then the induced subgraph $H'=G[C\cup\{u_1,u_2\}]$ is $\minusjedna_{H'}$-correspondence colorable by Lemma~\ref{lemma__33vamp},
    contradicting Observation~\ref{obs__forbidden_d1_corr_col_ind_subgraph}.
\end{proof}

In the probabilistic part of the argument, the vertices $v\in C$ with at least $2+\deg v-\Delta$ $(C,7)$-loose neighbors
(i.e., with at least two $(C,7)$-loose neighbors if $\deg v=\Delta$, and with at least one $(C,7)$-loose neighbor if $\deg v=\Delta-1$)
have a reasonable probability of ending up with a positive excess.  Indeed, it turns out that each of their $(C,7)$-loose neighbors has a decent chance
of belonging to the colored subgraph and blocking at $v$ the same color as some vertex of $C$.  If this occurs for at least two vertices $v\in C$, then we can extend the partial coloring to $C$ greedily.
However, Lemma~\ref{lemma__no_two_vertices_can_see_together_many_vertices_in_max_clique} only ensures that $v$ has
at least $p+\deg v-\Delta$ $(C,7)$-loose neighbors.  This is not enough when $p=1$, i.e., the core $C$ has size $\Delta-1$.
Thus, we need a stronger result to deal with this special case.

To prove this stronger result, it will be necessary to employ a stronger reduction procedure, where we not only delete some of the vertices
of a minimal $\Delta$-counterexample, but also add a few edges.  However, we need to be a bit careful not to create a clique of size $\Delta$ by this addition.
Let $G$ be a graph such that $\omega(G)\le \Delta-1$.  We say that a vertex $v'$ of $G$ is a \emph{$\Delta$-troublemaker} for another vertex $v$
if $G$ contains a clique of size $\Delta-2$ whose vertices are common neighbors of $v$ and $v'$.
Note that since $\omega(G)\le \Delta-1$, we have $v\nsim v'$; and in order to avoid increasing the clique number of $G$ to $\Delta$, we need to avoid adding edges between $\Delta$-troublemakers.
To aid us with this, let us show the following results on inclusionwise-maximal cliques in minimal $\Delta$-counterexamples, which we are also going to need in Section~\ref{section_04}.
For a real number~$r$, let $\KK_r(G)$ be the system of all inclusionwise-maximal cliques in $G$ whose size is at least $r$.

\begin{lemma} \label{lemma__interssize}
  Let $\Delta$ be a positive integer, let $r=\tfrac{3}{4}\Delta+1$, and let $(G,L,\Pi)$ be a minimal $\Delta$-counterexample.
  For all distinct cliques $A,B\in \KK_r(G)$, if $V(A)\cap V(B)\neq \emptyset$, then $|V(A\cap B)|\ge \tfrac{1}{2}\Delta+1$.
\end{lemma}
\begin{proof}
    For a vertex $v\in V(A\cap B)$, we have
    $$\Delta\ge \deg v\ge |V(A\cup B)|-1=|V(A)|+|V(B)|-|V(A\cap B)|-1\ge \tfrac{3}{2}\Delta+1-|V(A\cap B)|,$$
    and thus $|V(A\cap B)|\ge \tfrac{1}{2}\Delta+1$.
\end{proof}

\begin{lemma} \label{lemma__no_three_intersecting_cliques}
    Let $\Delta\ge 20$ be an integer, let $r=\tfrac{3}{4}\Delta+1$, and let $(G,L,\Pi)$ be a minimal $\Delta$-counterexample.
    Then every clique in $\KK_r(G)$ intersects at most one other clique in $\KK_r(G)$.
\end{lemma}

\begin{proof}
    Suppose for a contradiction that a clique $C\in \KK_r(G)$ intersects two distinct cliques $D,D'\in \KK_r(G)$. By Lemma~\ref{lemma__interssize}, we
    have $|V(C\cap D)|\ge \tfrac{1}{2}\Delta+1$ and $|V(C\cap D')|\ge \tfrac{1}{2}\Delta+1$. Since $|V(C)|\le \Delta-1$, this implies that $V(C \cap D \cap D')\neq \emptyset$.
    Let $v$ be any vertex of $C \cap D \cap D'$, let $B = G[(C \cup D \cup D') \setminus \{v\}]$ and let $H = B \vee v = G[C \cup D \cup D']$.

    Note that $\omega(H) \leq |V(H)| - 2$, and thus also $\omega(B)\le |V(B)|-2$.
    Indeed, if we had $\omega(H)\ge |V(H)|-1$, then $H$ would be obtained from a clique by deleting some of the edges incident with one of its vertices. 
    However, then $H\le G$ would contain at most two inclusionwise-maximal cliques, which is not possible since $C, D, D'\subseteq H$ are three distinct inclusionwise-maximal cliques in $G$.

    Since the vertex $v$ is universal in $H$ and each vertex of $H$ belongs to a clique of size at least $r$, we have $|V(H)| \leq \Delta + 1$ and $\delta(H) \geq r-1 \geq \frac{3}{4}\Delta$.
    Consequently, $|V(B)| \leq \Delta$ and $\delta(B) \geq \frac{3}{4} \Delta - 1 \geq \frac{1}{2}(|V(B)|+7)$.
    Therefore, by Lemma~\ref{lemma__near_clique}, the induced subgraph $H$ of $G$ is $\minusjedna_H$-correspondence colorable; this contradicts Observation~\ref{obs__forbidden_d1_corr_col_ind_subgraph}.
\end{proof}

The following claim then easily follows.

\begin{corollary}\label{cor__troublemaker_safe}
  Let $\Delta\ge 20$ be an integer.  If $(G,L,\Pi)$ is a minimal $\Delta$-counterexample, then the following claims hold.
  \begin{itemize}
  \item There is at most one $\Delta$-troublemaker for every vertex of $G$.
  \item Let $v_{1,1},v_{1,2},v_{2,1},v_{2,2}\in V(G)$ be distinct vertices such that $v_{1,1}\nsim v_{1,2}$ and $v_{2,1}\nsim v_{2,2}$.
  If $v_{1,1}$ is not a $\Delta$-troublemaker for $v_{1,2}$ and $v_{2,1}$ is not a $\Delta$-troublemaker for $v_{2,2}$,
  then $\omega(G+\{v_{1,1}v_{1,2},v_{2,1}v_{2,2}\})\le \Delta-1$.
  \end{itemize}
\end{corollary}
\begin{proof}
Let $r=\tfrac{3}{4}\Delta+1$ and note that every clique of size $\Delta-1$ in $G$ is contained in $\KK_r(G)$.

For the first part, suppose for a contradiction that $v_1$ and $v_2$ are two distinct $\Delta$-troublemakers for a vertex $v\in V(G)$.
For $i\in[2]$, let $K_i\subseteq G$ be a clique on $\Delta-2$ common neighbors of $v$ and $v_i$.
Since $\deg v\le \Delta<2(\Delta-2)$, the cliques $K_1$ and $K_2$ intersect. Moreover, since $v_1\nsim v$ and $v_2\nsim v$,
we have $v_1,v_2\not\in V(K_1\cup K_2)$.  Thus, the clique $K_1\join v\in \KK_r(G)$
intersects distinct cliques $K_1\join v_1,K_2\join v_2\in \KK_r(G)$, contradicting Lemma~\ref{lemma__no_three_intersecting_cliques}.

For the second part, suppose for a contradiction that $G+\{v_{1,1}v_{1,2},v_{2,1}v_{2,2}\}$ contains a clique $K$ of size $\Delta$.  Since $v_{1,1}$ is not a $\Delta$-troublemaker
for $v_{1,2}$ and $v_{2,1}$ is not a $\Delta$-troublemaker for $v_{2,2}$, we necessarily have $v_{1,1},v_{1,2},v_{2,1},v_{2,2}\in V(K)$.  For $i,j\in\{1,2\}$,
let $K_{i,j}$ be a clique in $G$ containing $v_{1,i}$ and $v_{2,j}$ of largest possible size; clearly, $|V(K_{i,j})|\ge |V(K)|-2=\Delta-2\ge r$,
and thus $K_{i,j}\in \KK_r(G)$.  Moreover, the four cliques $K_{1,1}$, $K_{1,2}$, $K_{2,1}$, and $K_{2,2}$ distinct,
since $v_{1,1}\nsim v_{1,2}$ and $v_{2,1}\nsim v_{2,2}$.  However, the cliques $K_{1,2}$ and $K_{2,1}$ both intersect the clique $K_{1,1}$, contradicting Lemma~\ref{lemma__no_three_intersecting_cliques}.
\end{proof}

As a final preparatory step, we need the following observation on correspondence coloring of cliques.
\begin{lemma}\label{lemma__colocliq}
Let $H$ be a clique of size $k\ge 1$ and let $(L,\Pi)$ be a $(k-1)$-correspondence assignment for~$H$.
If $H$ is $(L,\Pi)$-colorable, then at least one of the following claims holds:
\begin{itemize}
\item[(a)] There exists a vertex $v\in V(H)$ such that $|L(v)|\ge k$, %or
\item[(b)] there exist distinct vertices $z,v\in V(H)$ and a color $\alpha\in L(z)$ that does not block any color in $L(v)$ at $v$, or
\item[(c)] there exist distinct vertices $z_1,z_2,v\in V(H)$ and colors $\alpha_1\in L(z_1)$ and $\alpha_2\in L(z_2)$ 
such that $\{(z_1,\alpha_1),(z_2,\alpha_2)\}\not\in \Pi(z_1z_2)$ and the colors $\alpha_1$ at $z_1$ and $\alpha_2$ at $z_2$ together block at most one
color in $L(v)$ at $v$.
\end{itemize}
\end{lemma}
\begin{proof}
Suppose that neither (a) nor (b) holds, and thus $(L,\Pi)$ is a full exact $(k-1)$-correspondence assignment.
Let $v$ be an arbitrary vertex of $H$ and let $T$ be the star with the vertex set $V(H)$ and the center $v$.
By Observation~\ref{obs__straightening}, we can assume that all vertices have the same list $\Gamma$ and that $T$ is straight in the correspondence assignment $(L,\Pi)$.

By the assumptions, $H$ has an $(L,\Pi)$-coloring $\varphi$.  Since $|\Gamma|=k-1<|V(H)|$,
there exist distinct vertices $z_1,z_2\in V(H)$ such that $\varphi(z_1)=\varphi(z_2)$.
Since the edges incident with $v$ are straight and the correspondence assignment $(L,\Pi)$ is full, we have $z_1\neq v\neq z_2$;
and moreover, the colors $\varphi(z_1)$ and $\varphi(z_2)$ together block only one color at $v$, showing that (c) holds.
\end{proof}

We are now ready to restrict the neighborhoods in cores of size $\Delta-1$.

\begin{lemma} \label{lemma__no_bad_vertex_in_D-1_clique}
    Let $\Delta\ge 20$ be an integer and let $(G,L,\Pi)$ be a minimal $\Delta$-counterexample.
    Let $D$ be a $\Delta$-isolated near-clique in $G$ and let $C$ be a core of $D$.
    If $|C|=\Delta-1$, then every vertex $u\in V(G)\setminus C$ has at most $85$ neighbors in $C$.
\end{lemma}
    \begin{proof}
        Let $U$ be the set of neighbors of $u$ in $C$, and for contradiction suppose that $|U|\ge 86$.
	Since $\omega(G)\le \Delta-1$, the vertex $u$ has a non-neighbor $u'$ in $C$.

	\begin{claim}\label{cl__atmostone}
	At most one vertex of $U$ has degree $\Delta-1$.
	\end{claim}
	\begin{subproof}
	Suppose for a contradiction that distinct vertices $v_1,v_2\in U$ have degree $\Delta-1$.
	Since $(G,L,\Pi)$ is a minimal $\Delta$-counterexample, there exists a partial $(L,\Pi)$-coloring $\psi_1$ of $G$
	assigning colors to all vertices not in $C\cup \{u\}$.  Note that $v_1$ and $v_2$ have $\psi_1$-excess at least $0$.
	Moreover, $\deg^{-\psi_1} v_2=|C|$, $|L_{\psi_1}(u')|\ge \deg^{-\psi_1} u'-1=|C|-2$, and $|L_{\psi_1}(u)|\ge |U|-1$,
	and thus
	$$|L_{\psi_1}(u')|+|L_{\psi_1}(u)|>\deg^{-\psi_1} v_2.$$
	By Observation~\ref{obs__blocking_at_most_one_color}, we can extend $\psi_1$ to a partial $(L, \Pi)$-coloring $\psi$
	by choosing colors for $u$ and $u'$ so that the $\psi$-excess of $v_2$ is positive.
	By Observation~\ref{obs__greedy}, we can extend $\psi$ to an $(L,\Pi)$-coloring of $G$, coloring $v_1$ and $v_2$ last.
	This is a contradiction.
	\end{subproof}

	Let $S$ be the set of vertices in $V(G)\setminus (C\cup \{u\})$ with at least one neighbor in $U$.
	Since $|C|=\Delta-1$ and each vertex of $U$ is adjacent to $u$, each vertex of $U$ of degree $\Delta$ has exactly one neighbor in $S$.
	On the other hand, since the vertex $u$ is a $(C,7)$-tight neighbor of each vertex of $U$, Lemma~\ref{lemma__no_two_vertices_can_see_together_many_vertices_in_max_clique}
	implies that each vertex of $S$ has at most $7$ neighbors in $U$. Thus, by Claim~\ref{cl__atmostone}, we have
	$$|S|\ge \lceil \tfrac{|U|-1}{7}\rceil\ge 13.$$
	For each vertex $s\in S$, choose a neighbor $s'\in U$ arbitrarily.  Since each vertex of $U$ has at most one neighbor in $S$,
	$\{ss':s\in S\}$ is a matching, and $s$ is the unique neighbor of $s'$ in $S$.  Let $S'=\{s':s\in S\}$.
	Let $F$ be the auxiliary graph obtained from $G[S]$ by adding an edge between any distinct vertices $x,y\in S$ such that $y$ is a $\Delta$-troublemaker for $x$.
        \begin{claim}\label{cl__nomatch}
            The complement of $F$ does not contain any matching of size two.
        \end{claim}
	\begin{subproof}
	Suppose for a contradiction that $x_1,x_2,x_3,x_4\in S$ are distinct vertices such that $x_1\nsim_F x_2$ and $x_3\nsim_F x_4$.
	Let $T$ be the tree consisting of the path $x'_1x'_2x'_3x'_4$ and the edges $x_ix'_i$ for $i\in [4]$.
	By Observation~\ref{obs__straightening}, we can assume that $T$ is straight in the correspondence assignment $(L, \Pi)$,
	and moreover, all the vertices of $T$ have the same list $\Gamma$ of size $\Delta-1$.

	Let $G'$ be the graph obtained from $G-(C\cup\{u\})$ by adding the edges $x_1x_2$ and $x_3x_4$.
	Let $(L,\Pi')$ be the correspondence assignment obtained from $(L,\Pi)$ by
	defining $\Pi'(x_ix_{i+1})=\{\{(x_i,\gamma),(x_{i+1},\gamma)\}:\gamma\in\Gamma\}$ for $i\in\{1,3\}$.
	Thus, $(L,\Pi')$ is a full exact $(\Delta-1)$-correspondence assignment and the edges $x_1x_2$ and $x_3x_4$ are straight.
	Clearly $\Delta(G')\le \Delta$, and by Corollary~\ref{cor__troublemaker_safe} and the choice of $F$, we have $\omega(G')\le \Delta-1$.
	Since $(G,L,\Pi)$ is a minimal $\Delta$-counterexample, the graph $G'$ has an $(L,\Pi')$-correspondence coloring~$\psi$,
	which we view as a partial $(L,\Pi)$-coloring of $G$.

	Note that $x_2$ is the only neighbor of $x'_2$ in $\dom(\psi)$ and the edge $x_2x'_2$ is straight
	in the full correspondence assignment $(L,\Pi)$, and thus $L_\psi(x'_2)=\Gamma\setminus \psi(x_2)$.
	Since the edge $x_1x_2$ is straight in the correspondence assignment $(L,\Pi')$, we have $\psi(x_1)\neq\psi(x_2)$,
	and thus $\psi(x_1)\in L_\psi(x'_2)$.  By symmetric arguments we also have
	$\psi(x_2)\in L_\psi(x'_1)$, $\psi(x_3)\in L_\psi(x'_4)$, and $\psi(x_4)\in L_\psi(x'_3)$.

        Let us first try to color the vertex $x'_2$ by the color $\psi(x_1)\in L_\psi(x'_2)$
	and the vertex $x'_4$ by the color $\psi(x_3)\in L_\psi(x'_4)$.  Suppose that this succeeds, i.e., that the color $\psi(x_1)$ at $x'_2$ does not block the color $\psi(x_3)$ at $x'_4$,
	and let $\psi_1$ be the corresponding extension of $\psi$.  Since the edges of $T$ are straight in the correspondence assignment $(L,\Pi)$, this ensures that the vertices $x'_1$ and $x'_3$
	have non-negative $\psi_1$-excess.  Note that $|L_{\psi_1}(u')|\ge |C|-4$, $|L_{\psi_1}(u)|\ge |U|-3$, and $\deg^{-\psi}(x'_3)=|C|-2$,
	and thus
	\begin{equation}
	|L_{\psi_1}(u')|+|L_{\psi_1}(u)|>\deg^{-\psi_1} x'_3.\label{eq__canenb}
	\end{equation}
	By Observation~\ref{obs__blocking_at_most_one_color}, we can extend $\psi_1$ to a partial $(L, \Pi)$-coloring $\psi_2$
        by choosing colors for $u$ and $u'$ so that the $\psi_2$-excess of $x'_3$ is positive.
        By Observation~\ref{obs__greedy}, we can extend $\psi_2$ to an $(L,\Pi)$-coloring of~$G$, coloring the vertices $x'_1$ and $x'_3$ last.
        This is a contradiction.

	Therefore, we can assume that the color $\psi(x_1)$ at $x'_2$ blocks the color $\psi(x_3)$ at $x'_4$.
	Symmetrically, we can also assume that for every $i\in \{1,2\}$ and $j\in\{3,4\}$,
	the color $\psi(x_i)$ at $x'_{3-i}$ blocks the color $\psi(x_j)$ at $x'_{7-j}$.

	Let us extend $\psi$ to a partial coloring $\psi_3$ by giving $x'_1$ the color $\psi(x_2)$ and $x'_2$ the color $\psi(x_1)$
	(this is possible, since $\psi(x_1)\neq\psi(x_2)$ and the edge $x'_1x'_2$ is straight).  Note that this blocks only the color $\psi(x_4)$ at $x'_3$
	and only the color $\psi(x_3)$ at $x'_4$, and thus $x'_3$ and $x'_4$ have non-negative $\psi_3$-excess.
	A calculation analogous to (\ref{eq__canenb}) now shows that by Observation~\ref{obs__blocking_at_most_one_color}, we can extend $\psi_3$ to a partial $(L, \Pi)$-coloring $\psi_4$
        by choosing colors for $u$ and $u'$ so that the $\psi_4$-excess of $x'_3$ is positive.
        By Observation~\ref{obs__greedy}, we can extend $\psi_4$ to an $(L,\Pi)$-coloring of $G$, coloring the vertices $x'_4$ and $x'_3$ last.
        This is again a contradiction.
	\end{subproof}

	Since $\overline{F}$ does not contain any matching of size two, we have $\omega(F)\ge |V(F)|-2=|S|-2\ge 11$.
	By Corollary~\ref{cor__troublemaker_safe} and the construction of $F$, the set $E(F)\setminus E(G[S])$ forms a matching,
	and thus $\omega(G[S])\ge \lceil \omega(F)/2\rceil\ge 6$.

	Let $K\subseteq S$ be a set of $6$ vertices inducing a clique in $G$.
	Since $(G,L,\Pi)$ is a minimal $\Delta$-counterexample, the graph $G-(C\cup \{u\})$ has an $(L,\Pi)$-coloring $\psi_0$.
	Let $\psi$ be the partial $(L,\Pi)$-coloring of $G$ obtained from $\psi_0$ by uncoloring the vertices of $K$.
	Thus, the clique $G[K]$ is $(L_\psi,\Pi_\psi)$-colorable.  Moreover, since $(L_\psi,\Pi_\psi)$ is a $\minusjedna$-correspondence assignment
	for the graph $G'=G[C\cup K\cup\{u\}]$ and each vertex $x\in K$ has a neighbor $x'\not\in K$ in $G'$, the assignment $(L_\psi,\Pi_\psi)$ restricted to the clique $G[K]$
	is a $5$-correspondence assignment.  Hence, the claim (a), (b), or (c) from the statement of Lemma~\ref{lemma__colocliq} holds.
	Let $v\in K$ be the vertex corresponding to this claim, and let $z_1,z_2\in K\setminus\{v\}$ be distinct vertices
	chosen arbitrarily if (a) holds, chosen so that $z_1=z$ and $z_2$ is arbitrary if (b) holds, and chosen as the two vertices mentioned
	in the claim if (c) holds.  In all three cases, we can extend $\psi$ to a partial $(L,\Pi)$-coloring $\psi_1$ by choosing the colors of vertices $z_1$ and $z_2$ so that $|L_{\psi_1}(v)|\ge |K\setminus \{z_1,z_2\}|$.

	Let $u_1$ and $u_2$ be arbitrarily chosen distinct vertices of $K\setminus \{v,z_1,z_2\}$.
	Note that $|L_{\psi_1}(z'_1)|\ge |C|-1$, $|L_{\psi_1}(u_1)|\ge |K|-3=3$, and $\deg^{-\psi_1} u'_1=|C|+1$, and thus
	$$|L_{\psi_1}(z'_1)|+|L_{\psi_1}(u_1)|>\deg^{-\psi_1} u'_1-1.$$
	By Observation~\ref{obs__blocking_at_most_one_color}, we can extend $\psi_1$ to a partial $(L, \Pi)$-coloring $\psi_2$
        by choosing colors for $u_1$ and $z'_1$ so that the $\psi$-excess of $u'_1$ is non-negative.

	Similarly, note that $|L_{\psi_2}(z'_2)|\ge |C|-2$, $|L_{\psi_2}(u_2)|\ge |K|-4\ge 2$, and $\deg^{-\psi_2} u'_2=|C|$, and thus
	$$|L_{\psi_2}(z'_2)|+|L_{\psi_2}(u_2)|>\deg^{-\psi_2} u'_2-1.$$
	By Observation~\ref{obs__blocking_at_most_one_color}, we can extend $\psi_2$ to a partial $(L, \Pi)$-coloring $\psi_3$
        by choosing colors for $u_2$ and $z'_2$ so that the $\psi_3$-excess of $u'_2$ is non-negative.

	Next, let us extend $\psi_3$ to a partial partial $(L, \Pi)$-coloring $\psi_4$ by coloring the vertices of $K'=K\setminus \{u_1,u_2,z_1,z_2\}$ greedily,
	assigning the color to $v$ last.  This is possible, since $(L_{\psi_3},\Pi_{\psi_3})$ restricted to the clique $G[K']$
	is a $(|K'|-1)$-correspondence assignment and $|L_{\psi_3}(v)|\ge |K'|$ by the choice of $\psi_1$.

	Now note that $|L_{\psi_4}(u')|\ge |C|-4$, $|L_{\psi_4}(u)|\ge |U|-3\ge 83$, and $\deg^{-\psi_4} u'_2=|C|-2$, and thus
        $$|L_{\psi_4}(u')|+|L_{\psi_4}(u)|>\deg^{-\psi_4} u'_2.$$
        By Observation~\ref{obs__blocking_at_most_one_color}, we can extend $\psi_4$ to a partial $(L, \Pi)$-coloring $\psi_5$
        by choosing colors for $u$ and $u'$ so that the $\psi_5$-excess of $u'_2$ is positive.

	Finally, we can extend $\psi_5$ to an $(L,\Pi)$-coloring of $G$ by Observation~\ref{obs__greedy}, coloring the vertices $u'_1$ and~$u'_2$ last.
	This is a contradiction.
    \end{proof}

Using Lemmas~\ref{lemma__no_two_vertices_can_see_together_many_vertices_in_max_clique} and \ref{lemma__no_bad_vertex_in_D-1_clique},
we can now prove the second key result of this section, which essentially states that for a core of any large $\Delta$-isolated near-clique, there are many disjoint opportunities to
obtain vertices with positive excess in the probabilistic part of the argument.
Let $G$ be a graph and let $C$ and $X$ be subsets of its vertices.
For positive integers $m$ and $\Delta$, we say that $X$ is a \emph{$(\Delta,m)$-benefactor for $C$} if
\begin{itemize}
\item $X\cap C$ consists of a single vertex $v$,
\item the vertices in $X\setminus \{v\}$ are $(C,m)$-loose neighbors of $v$, and
\item $|X\setminus\{v\}|=2+\deg v-\Delta$.
\end{itemize}

\begin{theorem} \label{thm__disjoint_triples}
    Let $\Delta$ be a positive integer and let $(G,L,\Pi)$ be a minimal $\Delta$-counterexample.
    Let $D$ be a $\Delta$-isolated near-clique in $G$, let $C$ be a core of $D$, and let $p=\Delta-|C|$.
    If $\Delta\ge 4p+16$, then there exists a set $\WW$ of at least $\tfrac{1}{100}\Delta$ pairwise
    disjoint $(\Delta,85)$-benefactors for $C$.
\end{theorem}
\begin{proof}
    Let $\WW$ be an inclusionwise-maximal set of pairwise disjoint $(\Delta,85)$-benefactors for $C$.
    Let $k=|\WW|$, $S=\bigl(\bigcup \WW\bigr)\setminus C$, $T=C\cap\bigcup \WW$,
    and for each $X\in \WW$, let $b(X)$ denote the number of edges of $G$ with one endpoint in $X\setminus C$ and the other endpoint in $C\setminus T$.
    Note that $|X\setminus C|\le 2$, the vertices of $X\setminus C$ are $(C,85)$-loose, and by Lemma~\ref{lemma__no_two_vertices_can_see_together_many_vertices_in_max_clique},
    at most one of them is $(C,7)$-tight.  Moreover, each of them is adjacent to the vertex contained in $C\cap X$.  Consequently,
    $$b(X)\le 90.$$

    Let us count the number $\beta$ of edges between $C \setminus T$ and $S$ in two different ways.
    Consider a vertex $v\in C\setminus T$, and note that $v$ has $p+1+\deg v-\Delta$ neighbors outside of $C$.
    By Lemmas~\ref{lemma__no_two_vertices_can_see_together_many_vertices_in_max_clique} and \ref{lemma__no_bad_vertex_in_D-1_clique},
    at least $\max(p,2)+\deg v-\Delta$ of them are $(C,85)$-loose.
    By the maximality of $\WW$, if $\deg v=\Delta-1$, then all $(C,85)$-loose neighbors of $v$ belong to $S$,
    and if $\deg v=\Delta$, then all but at most one of them do, as otherwise $v$ together with $2+\deg v-\Delta$ of its neighbors
    would form a $(\Delta,85)$-benefactor for $C$ disjoint from those in $\WW$.  In either case, we see that $v$ has at least
    $\max(p-1,1)$ neighbors in $S$.  Therefore,
    $$\beta\ge \max(p-1,1)\cdot (|C|-|T|)=\max(p-1,1)\cdot (\Delta-p-k).$$
    On the other hand, we have
    $$\beta=\sum_{X\in \WW} b(X)\le 90k.$$
    By combining the inequalities, we get
    $$k\ge \frac{\max(p-1,1)\cdot (\Delta-p)}{\max(p-1,1)+90}.$$
    Since $\Delta\ge 4p+16$, this implies that $k\ge \tfrac{1}{100}\Delta$.
\end{proof}

% Sparse-dense decomposition
\newpage 
\section{Sparse-dense decomposition} \label{section_04}

We now construct a \emph{sparse-dense decomposition} of a minimal $\Delta$-counterexample, partitioning its vertex set into large $\Delta$-isolated near-cliques
and a set of \emph{sparse} vertices, i.e., vertices with many non-edges among its neighbors.
More precisely, for a real number $\sigma$, a vertex $v$ of a graph $G$ is \emph{$\sigma$-sparse} if $G[N(v)]$ has at least $\sigma$ non-edges,
and \emph{$\sigma$-dense} otherwise. 
\begin{definition} \label{def__sparse-dense_decomposition}
    Let $\sigma$ and $\rho$ be positive real numbers.
    A \emph{sparse-dense $(\Delta,\sigma,\rho)$-decomposition} of a graph $G$ is a set $\DD$ of pairwise vertex-disjoint
    $\Delta$-isolated near-cliques in $G$, such that
    \begin{itemize}
    \item[(a)] each near-clique $D\in\DD$ satisfies $|D|>\Delta-\rho$ and $\delta(G[D])\ge \tfrac{3}{4}\Delta$, and
    \item[(b)] all vertices in $V(G)\setminus \bigcup \DD$ are $\sigma$-sparse.
    \end{itemize}
\end{definition}

\noindent The near-cliques of the sparse-dense decomposition arise from the following lemma.

\begin{lemma} \label{lemma__union_of_two_cliques_is_near_clique}
    Let $\Delta\ge 10$ be an integer, let $r=\tfrac{3}{4}\Delta+1$, and let $(G,L,\Pi)$ be a minimal $\Delta$-counterexample.
    For distinct cliques $C, C' \in \KK_r(G)$, if $C \cap C' \neq \emptyset$, then $V(C\cup C')$ is a near-clique in $G$.
\end{lemma}
\begin{proof}
    By Lemma~\ref{lemma__interssize}, we have $|V(C\cap C')|\ge \tfrac{1}{2}\Delta+1\ge 6$, and thus $H=G[V(C\cup C')]$ has at least $6$ universal vertices.
    By Observation~\ref{obs__forbidden_d1_corr_col_ind_subgraph}, the induced subgraph $H$ of $G$ is not $\minusjedna_H$-correspondence colorable,
    which by Lemma~\ref{lemma__d1_corr_col} implies that $\omega(H)\ge |V(H)|-1$.  Therefore, $V(C\cup C')$ is indeed a near-clique in $G$.
\end{proof}

\noindent Next, let us rephrase Theorem~\ref{thm__many_nones} in an easier to use form.

\begin{lemma}\label{lemma__fewednei}
  Let $\omega\ge 13$ be a real number, let $\Delta\ge \omega+1$ be an integer, let $(G,L,\Pi)$ be a minimal $\Delta$-counterexample,
  let $v$ be a vertex of $G$, and let $H=G[N(v)]$.  If $\omega(H)<\omega$,
  then $H$ has more than $$\frac{(\Delta-\omega-2)(\Delta+\omega-15)}{4}$$ non-edges.
\end{lemma}
\begin{proof}
    By Observation~\ref{obs__mindeg}, the graph $H$ has at least $\Delta-1$ vertices, and since $\omega(H)<\omega\le \Delta-1$,
    the graph $H$ is not a clique.
    By Observation~\ref{obs__forbidden_d1_corr_col_ind_subgraph}, for every $B\le H$, the graph $B\vee v$ (which is an induced subgraph of $G$)
    is not $\minusjedna$-correspondence colorable.
    Therefore, Theorem~\ref{thm__many_nones} implies that $H$ has at least
    \begin{align*}
    \frac{(|V(H)|-\omega(H)-1)(|V(H)|+\omega(H)-14)}{4}&\ge \frac{(\Delta-\omega(H)-2)(\Delta+\omega(H)-15)}{4}\\
    &>\frac{(\Delta-\omega-2)(\Delta+\omega-15)}{4}
    \end{align*}
    non-edges.  Note that the last inequality uses the fact that the function $f(x)=(\Delta-x-2)(\Delta+x-15)$
    is increasing for $x<6.5$, decreasing for $x> 6.5$, $f(0)=f(13)$, and $\omega\ge 13$.
\end{proof}

We are now ready to establish the existence of sparse-dense decompositions of minimal $\Delta$-counter\-examples.

\begin{theorem} \label{thm__sparse-dense_decomposition}
    Let $\Delta\ge 40$ be an integer and let $\sigma>0$ and $\rho$ be real numbers such that $\tfrac{4\sigma}{\Delta}+2\le\rho\le \tfrac{1}{4}\Delta$.
    If $(G,L,\Pi)$ is a minimal $\Delta$-counterexample, then $G$ has a sparse-dense $(\Delta,\sigma,\rho)$-decomposition.
\end{theorem}
\begin{proof}
    Let $r=\tfrac{3}{4} \Delta+1$.
    Let $I$ be the intersection graph of $\KK_r(G)$,
    i.e., the auxiliary graph with vertex set $\KK_r(G)$ and an edge between two vertices if and only if the corresponding cliques have nonzero intersection.
    For each component $Q$ of $I$, let $D_Q=\bigcup \{V(C):C\in V(Q)\}$; note that the component $Q$ has size at most two by Lemma~\ref{lemma__no_three_intersecting_cliques},
    and that $D_Q$ is a near-clique in $G$ by Lemma~\ref{lemma__union_of_two_cliques_is_near_clique}.
    Let
    $$\DD=\{D_Q:\text{$Q$ is a component of $I$ such that $D_Q$ contains a $\sigma$-dense vertex}\},$$
    and let $S=V(G)\setminus \bigcup\DD$.
    We claim that $\DD$ is a sparse-dense $(\Delta,\sigma,\rho)$-decomposition of $G$.
    We need to verify that the conditions (a) and (b) from Definition~\ref{def__sparse-dense_decomposition} hold and that the near-cliques in $\DD$ are $\Delta$-isolated.

    Consider any vertex $v\in S$.  If $v\in D_Q$ for some component $Q$ of $I$, then the vertex $v$ is $\sigma$-sparse by the choice of $\DD$.
    Otherwise, note that the graph $H=G[N(v)]$ satisfies $\omega(H)<r-1=\tfrac{3}{4}\Delta$, since $v$ is not contained in any (maximal) clique
    of size at least $r$.  Therefore, Lemma~\ref{lemma__fewednei} implies that $H$ has more than
    $$\frac{(\Delta/4-2)(7\Delta/4-15)}{4}=\frac{7\Delta^2-116\Delta+480}{64}>\frac{\Delta^2}{16}>\sigma$$
    non-edges, where the inequalities hold since $\Delta\ge 40$ and $\tfrac{4\sigma}{\Delta}+2\le \tfrac{1}{4}\Delta$. 
    It follows that $v$ is $\sigma$-sparse, and thus $S$ consists only of $\sigma$-sparse vertices.  Therefore, (b) holds.

    Consider now any near-clique $D\in \DD$.  Note that $G[D]$ is the union of at most two cliques from $\KK_r(G)$, and thus
    $\delta(G[D])\ge r-1=\tfrac{3}{4}\Delta$.  Moreover, by the choice of $\DD$, there exists a $\sigma$-dense vertex
    $v\in D$.  We claim that the graph $H=G[N(v)]$ satisfies $\omega(H)\ge \Delta-\rho$.
    Indeed, otherwise Lemma~\ref{lemma__fewednei} with $\omega=\Delta-\rho$ would imply that $H$ has more than
    $$\frac{(\rho-2)(2\Delta-\rho-15)}{4}\ge \frac{(\rho-2)\Delta}{4}\ge \sigma$$
    non-edges (the first inequality uses that $\rho+15\le \tfrac{1}{4}\Delta+15\le \Delta$), and thus $v$ would be $\sigma$-sparse.  It follows that $v$ is contained in a maximal clique $C$ of size at least
    $\Delta-\rho+1\ge r$.  This clique belongs to $\KK_r(G)$ and intersects $D$, and thus $V(C)\subseteq D$.  Therefore, $|D|\ge \Delta-\rho+1$.
    It follows that (a) holds.

    Finally, let us argue that each near-clique $D\in\DD$ in $\Delta$-isolated.
    Consider any vertex $z\in V(G)\setminus D$, and let $F=G[N(z)\cap D]$.  If $\omega(F)\ge \tfrac{3}{4}\Delta$, then
    $z$ would be contained in a maximal clique of size at least $\omega(F)+1\ge r$ intersecting $D$.  However, by the construction,
    the vertices of all such cliques are contained in $D$, contradicting $z\not\in D$.  Therefore, we have $\omega(F)<\tfrac{3}{4}\Delta$,
    and since $\omega(F)$ is an integer, it follows that $\omega(F)\le \lceil \tfrac{3}{4}\Delta\rceil-1$.
    On the other hand, the set $V(F)$ is a subset of the near-clique~$D$, and thus $\omega(F)\ge |V(F)|-1$.
    It follows that $z$ has $|V(F)|\le \omega(F)+1\le \lceil \tfrac{3}{4}\Delta\rceil$ neighbors in $D$.
    Consequently, the near-clique $D$ is $\Delta$-isolated.
\end{proof}

% Probability part
\newpage

\section{Probabilistic analysis} \label{section_05} % TO DO (+): Consider to change the title.

Let us now finish the proof of Theorem~\ref{thm__bk_conjecture_is_true_asym_for_DP_coloring} by a probabilistic argument.
Let us start by summarizing the results we need from the previous sections.
For a minimal $\Delta$-counterexample $(G,L,\Pi)$, a \emph{setting} is a triple $(\DD,\CC,\{\WW_C\}_{C\in \CC})$, where
\begin{itemize}
\item $\DD$ is a sparse-dense $\bigl(\Delta,\tfrac{1}{120}\Delta^2,\tfrac{1}{28}\Delta\bigr)$-decomposition of $G$,
\item $\CC=\{C_D:D\in\DD\}$, where for each $D\in \DD$, the set $C_D$ is a core of the near-clique $D$, and
\item for each clique $C\in \CC$, $\WW_C$ is a set of exactly $\lceil\tfrac{1}{100}\Delta\rceil$ pairwise disjoint $(\Delta,85)$-benefactors for $C$.
\end{itemize}

\noindent Theorems~\ref{thm__sparse-dense_decomposition} (with $\sigma=\tfrac{1}{120}\Delta^2$ and $\rho=\tfrac{1}{28}\Delta$) and \ref{thm__disjoint_triples} (with $p<\rho$) have the following consequence.
\begin{corollary}\label{cor__summary}
Let $\Delta\ge 40$ be an integer.  For every minimal $\Delta$-counterexample $(G,L,\Pi)$, there exists a setting.
\end{corollary}

\subsection{Transversal of the setting}

For a fixed sparse-dense decomposition, we first pick a subset of vertices
which (among other related properties) intersects every core and the
neighborhood of every sparse vertex in roughly $\tfrac{1}{5}\Delta$ vertices.  A
random subset with each vertex chosen independently with probability around $1/5$
has this property almost surely;
this is sufficient, by the use of the Lovász Local Lemma.

More precisely, let $\Delta$ be a positive integer and let $(\DD,\CC,\{\WW_C\}_{C\in \CC})$ be a setting for a minimal $\Delta$-counterexample $(G,L,\Pi)$.
For a core $C\in \CC$ and a set $X\in \WW_C$, we say that a set $S\subseteq V(G)$
\emph{$C$-bisects} $X$ if $X\cap S=X\setminus C$.
We say that the set $S$ is a \emph{transversal} for the setting if it satisfies the following conditions.
\begin{itemize}
\item[(i)] Each vertex $v\in V(G)$ has at most $\tfrac{1}{5}\Delta$ neighbors in $S$.
\item[(ii)] For each vertex $v\in V(G)\setminus\bigcup \DD$, at least $\tfrac{1}{4800}\Delta^2$ of the non-edges between the neigbors of $v$
have both endpoints in $S$.
\item[(iii)] For each $C\in \CC$, $|C\cap S|\ge \tfrac{1}{6}\Delta$.
\item[(iv)] For each $C\in \CC$, the set $S$ $C$-bisects at least $\tfrac{\Delta}{10^4}$ elements of $\WW_C$.
\end{itemize}

\begin{lemma}\label{lemma__extransvers}
Let $\Delta$ be a positive integer and let $(\DD,\CC,\{\WW_C\}_{C\in \CC})$ be a setting for a minimal $\Delta$-counterexample $(G,L,\Pi)$.
If $\Delta\ge 10^8$, then there exists a transversal for the setting.
\end{lemma}
\begin{proof}
Let $S$ be the set of vertices that selects each vertex of $G$ independently at random with probability $\tfrac{11}{60}=\tfrac{1}{5}-\tfrac{1}{60}$.
\begin{itemize}
\item[(i)] For each vertex $v\in V(G)$, let $S(v)$ be the number of neighbors of $v$ in $S$ and let $A_{v,\mathrm{i}}$ be the event that $S(v)>\tfrac{1}{5}\Delta$.
Note that $S(v)$ is a sum of independent Bernoulli variables and $\ex{S(v)}=\tfrac{11}{60}\deg v\le \tfrac{11}{60}\Delta$, and
thus a Chernoff Bound (Theorem~\ref{tool__chernoff} with $\lambda=\tfrac{11}{60}\Delta$ and $t=\tfrac{1}{60}\Delta$) implies
$$\Pr{A_{v,\mathrm{i}}}\le e^{-\frac{1}{1380}\Delta}\le \frac{1}{16\Delta^2}.$$
Let us also define $g(A_{v,\mathrm{i}})=N(v)$, denoting the set of vertices whose outcomes determine $A_{v,\mathrm{i}}$.
\item[(ii)] For each vertex $v\in V(G)\setminus\bigcup \DD$, let $S'(v)$ be the number of non-edges between the neighbors of $v$ in $S$
and let $A_{v,\mathrm{ii}}$ be the event that $S'(v)<\tfrac{1}{4800}\Delta^2$.  Let $M$ be the set of all non-edges between the neighbors of $v$ in $G$ and let $m=|M|$. We have $m\ge \tfrac{1}{120}\Delta^2$,
and thus $\Pr{A_{v,\mathrm{ii}}}\le\Pr{S'(v)<\tfrac{1}{40}m}$.  Moreover, $\ex{S'(v)}=(\tfrac{11}{60})^2m=\tfrac{1}{40}m+\tfrac{31}{3600}m$.  Let $u_1$, \ldots, $u_{\deg v}$ be the neighbors of $v$
and for $i\in[\deg v]$, let $m_i$ denote the number of non-edges in $M$ incident with $u_i$.
Note that
$$\sum_{i=1}^{\deg v} m_i^2\le\Delta\sum_{i=1}^{\deg v} m_i=2m\Delta.$$
The random variable $S'(v)$ is determined by the independent indicator random variables for the events $u_i\in S$, and for each $i\in[\deg v]$, changing whether $u_i$ belongs to $S$ changes the value of $S'(v)$
at most by $m_i$.  By McDiarmid's Inequality (Theorem~\ref{tool__mcdiarmid}),
\begin{align*}
\Pr{A_{v,\mathrm{ii}}}&\le \Pr{S'(v)<\ex{S'(v)}-\tfrac{31}{3600}m}\le \exp\Bigl(-\frac{961 m^2}{648\cdot 10^4\cdot 2m\Delta}\Bigr)\\
&\le \exp\Bigl(-\frac{961}{15552\cdot 10^5}\Delta\Bigr)\le \frac{1}{16\Delta^2}.
\end{align*}
Let us also define $g(A_{v,\mathrm{ii}})=N(v)$.
\item[(iii)] For each $C\in \CC$, let $A_{C,\mathrm{iii}}$ be the event that $|C\cap S|<\tfrac{1}{6}\Delta$.
$$\ex{|C\cap S|}=\tfrac{11}{60}|C|\ge \tfrac{11}{60}\bigl(\Delta-\tfrac{1}{28}\Delta\bigr)=\tfrac{99}{560}\Delta=\tfrac{1}{6}\Delta+\tfrac{17}{1680}\Delta.$$
A Chernoff Bound (Theorem~\ref{tool__chernoff} with $\eta=\tfrac{99}{560}\Delta$ and $t=\tfrac{17}{1680}\Delta$) gives
$$\Pr{A_{C,\mathrm{iii}}}\le e^{-\frac{289}{997920}\Delta}\le \frac{1}{16\Delta^2}.$$
Let us also define $g(A_{C,\mathrm{iii}})=C$.
\item[(iv)] For each $C\in \CC$, let $S(C)$ be the number of elements of $\WW_C$ $C$-bisected by $S$, and let $A_{C,\mathrm{iv}}$ be the event that $S(C)<\tfrac{\Delta}{10^4}$.
Let $w_C=|\WW_C|= \lceil \tfrac{1}{100}\Delta\rceil$; we have $\Pr{A_{C,\mathrm{iv}}}\le \Pr{S(C)<\tfrac{1}{100}w_C}$.
Note that
$$\ex{S(C)}\ge \frac{49}{60}\cdot \Bigl(\frac{11}{60}\Bigr)^2\cdot w_C=\frac{5929}{216\cdot 10^3}w_C=\frac{1}{100}w_C+\frac{3769}{216\cdot 10^3}w_C.$$
Moreover, for distinct elements of $\WW_C$, the events whether they are $C$-bisected by $S$ are independent.
Hence, a Chernoff Bound (Theorem~\ref{tool__chernoff} with $\eta=\frac{5929}{216\cdot 10^3}w_C$ and $t=\frac{3769}{216\cdot 10^3}w_C$)
gives
$$\Pr{A_{C,\mathrm{iv}}}\le e^{-\frac{14205361}{2561328\cdot 10^3}w_C}\le e^{-\frac{14205361}{2561328\cdot 10^5}\Delta}\le \frac{1}{16\Delta^2}.$$
Let us also define $g(A_{C,\mathrm{iv}})=\bigcup \WW_C$.
\end{itemize}
Let $\EE$ be the set of the events discussed in the four cases. For each $A\in \EE$, let $\EE_A$ be the set of events $A'\in \EE$ such that $g(A)\cap g(A')\neq\emptyset$.
Note that $|\EE_A|\le 4\Delta^2$.  Indeed, we have
$$|g(A)|\le \max\bigl(\{\deg v:v\in V(G)\}\cup \{|C|:C\in \CC\}\cup \{3w_C:C\in\CC\}\bigr)\le \Delta,$$
and $\EE_A$ contains
\begin{itemize}
\item events $A_{v,\mathrm{i}}$ for at most $|g(A)|\cdot \Delta\le \Delta^2$ vertices $v$ with a neighbor in $g(A)$,
\item events $A_{v,\mathrm{ii}}$ for at most $|g(A)|\cdot \Delta\le \Delta^2$ vertices $v$ with a neighbor in $g(A)$,
\item events $A_{C,\mathrm{iii}}$ for at most $|g(A)|\le \Delta$ cores $C\in \CC$ such that a vertex of $g(A)$ belongs to $C$, and
\item events $A_{C,\mathrm{iv}}$ for at most $|g(A)|+|g(A)|\cdot \Delta\le \Delta+\Delta^2$ cores $C\in\CC$ such that
either a vertex of $g(A)$ or a neighbor of a vertex of $g(A)$ belongs to $C$.
\end{itemize}

Observe that $A$ is mutually independent of $\EE\setminus\EE_A$, and thus for every set $\EE'\subseteq \EE\setminus \EE_A$,
$$\Pr{A\Bigm|\bigwedge_{A'\in \EE'} \lnot A'}=\Pr{A}\le \frac{1}{16\Delta^2}.$$
By  the Lovász Local Lemma (Theorem~\ref{tool__lovasz_local_lemma}), there exists a choice for the set $S$ such that none of the events in $\EE$ occur, and thus $S$ is a transversal for the setting.
\end{proof}

\subsection{Uniform coloring of the transversal}

Let $\Delta$ be a positive integer and let $G$ be a graph of maximum degree at most $\Delta$.
Let $(L,\Pi)$ be a correspondence assignment for $G$, let $S$ be a set of vertices of $G$,
and let $\psi$ be an $(L,\Pi)$-coloring of~$G[S]$.
Let $v$ be a vertex of $G$ not belonging to $S$ and let $M$ be a set of non-edges of $G$ between the vertices of $N(v)\cap S$.
A non-edge $xy\in M$ is \emph{$(v,\psi)$-monochromatic} if $\psi(x)$ and $\psi(y)$ together block at most one color at $v$.
Our first goal will be to show that, under a uniformly random choice of $\psi$ subject to a few additional assumptions, there is a good chance that
there are at least $2+\deg v-\Delta$ $(v,\psi)$-monochromatic non-edges in $M$ (if that is the case, we say that $\psi$ \emph{$(\Delta,M)$-saves $v$}).
Assuming that $|L(v)|=\Delta-1$, this ensures that $v$ has a positive $\psi$-excess and $\psi$ (or even any extension of $\psi$) can be extended to $v$ greedily.

Our plan for proving this is to perform resampling (see Lemma~\ref{tool__resampling}) of $\psi$ at each non-edge $m\in M$, arguing that we are somewhat
likely to make $m$ $(v,\psi)$-monochromatic.  Of course, we would like to avoid turning previously processed $(v,\psi)$-monochromatic non-edges back into non-$(v,\psi)$-monochromatic ones
by the resampling.  Thus, we want to ignore $m$ when it shares an endpoint with such a $(v,\psi)$-monochromatic non-edge.  This brings a concern that
we could be forced to ignore too many edges of $M$, motivating the following property:  For a positive real number $\beta$, we say that the set $M$
is \emph{$(v,\Delta,\beta)$-rich} if for every set $M'\subseteq M$ of size less than $2+\deg v-\Delta$, there are at least $\beta$ non-edges in $M$ that do
not share an endpoint with any non-edge in $M'$.

As a further complication, we need to show that the claim holds somewhat independently of what happens in the rest of the graph,
so that we can apply the lopsided Lovász Local Lemma.  For this, we need to estimate the conditional probability that $\psi$ $(\Delta,M)$-saves $v$
given that the colors of non-neighbors of $v$ are fixed somehow.  More precisely, for a positive real number $q$, we say that \emph{the uniform $(L,\Pi)$-coloring of $G[S]$ $(\Delta,M)$-fails
$v$ with probability at most $q$} if for every partial $(L,\Pi)$-coloring $\sigma$ of $G[S]$ such that $N(v)\cap \dom(\sigma)=\emptyset$,
$$\Pr{\text{$\psi$ does not $(\Delta,M)$-save $v$}\mid\text{$\psi$ extends $\sigma$}}\le q.$$

\begin{lemma}\label{lemma__can_free_vertex}
Let $\Delta$ be a positive integer, let $\beta\ge 1$, $\gamma>0$, and $\delta>0$ be real numbers
and
$$q=(1+\beta p)e^{-(\beta-1)p}\le e(1+\beta p)e^{-\beta p},$$
where $p=\tfrac{\delta}{\Delta^2}$.
Let $G$ be a graph of maximum degree at most $\Delta$ and $(L,\Pi)$ a correspondence assignment for $G$ such that $|L(u)|\le \Delta$ holds for every $u\in V(G)$.
Let $S$ be a set of vertices of $G$, $v$ a vertex of $G$ not belonging to $S$, and $M$ a set of non-edges of $G$ between vertices of $N(v)\cap S$.
If
\begin{itemize}
\item[(a)] every vertex of $G$ has at most $\gamma$ neighbors in $S$,
\item[(b)] the set $M$ is $(v,\Delta,\beta)$-rich, and
\item[(c)] $|L(u)|\ge \tfrac{\Delta+4\gamma+\delta-4}{2}$ holds for every $u\in N(v)\cap S$,
\end{itemize}
then the uniform $(L,\Pi)$-coloring of $G[S]$ $(\Delta,M)$-fails $v$ with probability at most $q$.
\end{lemma}
\begin{proof}
We can assume that $\deg v\in\{\Delta-1,\Delta\}$, as otherwise every $(L,\Pi)$-coloring of $G[S]$ trivially $(\Delta,M)$-saves $v$.
Let $M=\{m_1, \ldots, m_c\}$ and for each $i\in [c]$, let $m_i=x_iy_i$ and $M_i=\{m_1,\ldots,m_i\}$.  Also, let $M_0=\emptyset$.

By Observation~\ref{obs__straightening}, we can assume that all edges of $G$ incident with $v$
are straight and that the lists of all vertices of $N(v)\cap S$ are subsets of the same set $\Gamma\supseteq L(v)$ of size $\Delta$.
For a partial $(L,\Pi)$-coloring $\varphi$ of $G[S]$ with $N(v)\cap S\subseteq \dom(\varphi)$, we say that a non-edge $m=xy\in M$ is \emph{$\varphi$-monochromatic} if $\varphi(x)=\varphi(y)$;
clearly, a $\varphi$-monochromatic non-edge is also $(v,\varphi)$-monochromatic.
For $i\in\{0,1,\ldots,c\}$, we define $\mu_i(\varphi)$ as the set of $\varphi$-monochromatic non-edges in $M_i$.
For $i\in [c]$, let $f_i$ be the subset selector defined as follows:
If $x_i$ or $y_i$ is incident with a $\varphi$-monochromatic non-edge belonging to $M_{i-1}$,
then $f_i(\varphi)=\emptyset$, otherwise $f_i(\varphi)=\{x_i,y_i\}$.

Let us fix any partial $(L,\Pi)$-coloring $\sigma$ of $G[S]$ such that $N(v)\cap \dom(\sigma)=\emptyset$.
Let $S'=S\setminus\dom(\sigma)$ and let $\psi'$ be a uniformly random $(L^\sigma,\Pi^\sigma)$-coloring of $G[S']$;
by Observation~\ref{obs__combine}, $\sigma\cup\psi'$ is a uniformly random $(L,\Pi)$-coloring of $G[S]$ that extends $\sigma$.
Therefore, to obtain the conclusion of the lemma, we need to show that
\begin{equation}\label{eq__toprove}
\Pr{\text{$\psi'$ does not $(\Delta,M)$-save $v$}}\le q.
\end{equation}
By Lemma~\ref{tool__resampling}, we can obtain the uniformly random $(L^\sigma,\Pi^\sigma)$-coloring $\psi'$ of $G[S']$ by the following process:
Let $\psi_0$ be a uniformly random $(L^\sigma,\Pi^\sigma)$-coloring of $G[S']$.
For $i\in[c]$, let $\psi_i$ be an $f_i$-invariant resampling of
$\psi_{i-1}$. Finally, let $\psi'=\psi_c$.

The choice of the subset selectors $f_i$ implies that $\mu_{i-1}(\psi_{i-1})\subseteq \mu_i(\psi_i)$ holds for each $i\in[c]$.
\begin{claim}\label{cl__probincr}
For each index $i\in[c]$, if $f_i(\psi_{i-1})=\{x_i,y_i\}$, then
$$\Pr{|\mu_i(\psi_i)|\ge |\mu_{i-1}(\psi_i)|+1}\ge p.$$
\end{claim}
\begin{subproof}
Recall that $f_i(\psi_{i-1})=\{x_i,y_i\}$ means that neither of the endpoints $x_i$ and $y_i$ of the non-edge $m_i$ is incident with a $\psi_{i-1}$-monochromatic non-edge belonging to $M_{i-1}$.
Let $X_i=\{\psi_{i-1}(z):zx_i\in M_{i-1}\}$ and $Y_i=\{\psi_{i-1}(z):zy_i\in M_{i-1}\}$;
since $|N(v)\cap S\setminus\{x_i,y_i\}|\le \gamma-2$, we have $|X_i|,|Y_i|\le \gamma-2$.
Let $\psi'_{i-1}$ be the restriction of $\psi_{i-1}$ to $S'\setminus\{x_i,y_i\}$,
and let $(L_i,\Pi_i)$ be the correspondence assignment for the induced subgraph $G[\{x_i,y_i\}]$
such that
\begin{align*}
L_i(x_i)&=L^{\sigma\cup \psi'_{i-1}}(x_i)\setminus X_i,\\
L_i(y_i)&=L^{\sigma\cup \psi'_{i-1}}(y_i)\setminus Y_i,
\end{align*}
and $\Pi_i$ is null (since $x_iy_i$ is a non-edge).
Observe that an $\{x_i,y_i\}$-modification of $\psi_{i-1}$ is $f_i$-invariant if and only if its restriction to $\{x_i,y_i\}$ is an $(L_i,\Pi_i)$-coloring of $G[\{x_i,y_i\}]$;
e.g., if we ended up recoloring the vertex $x_i$ to the color $\psi_{i-1}(z)\in X_i$ for some non-edge $zx_i\in M_{i-1}$, then the non-edge $zx_i$ would become $\psi_{i-1}$-monochromatic
and the value of $f_i$ would change to $\emptyset$.  Therefore, $\psi_i=\psi'_{i-1}\cup \psi''_i$, where $\psi''_i$ is a uniformly random $(L_i,\Pi_i)$-coloring of $G[\{x_i,y_i\}]$.
Since $\Delta(G[S])\le \gamma$ by (a) and $|X_i|,|Y_i|\le \gamma-2$, by (c) we have
$$|L_i(x_i)\cap L_i(y_i)|\ge |L(x_i)\cap L(y_i)|-4\gamma+4\ge |L(x_i)|+|L(y_i)|-|\Gamma|-4\gamma +4\ge \delta.$$
Thus,
\[\Pr{|\mu_i(\psi_i)|\ge |\mu_{i-1}(\psi_i)|+1}=\Pr{\psi_i(x_i)=\psi_i(y_i)}=\frac{|L_i(x_i)\cap L_i(y_i)|}{|L_i(x_i)|}\cdot \frac{1}{|L_i(y_i)|}\ge \frac{\delta}{\Delta^2}=p.\qedhere\]
\end{subproof}
We have
$$\Pr{\text{$\psi'$ does not $(\Delta,M)$-save $v$}}\le \Pr{|\mu_n(\psi_n)|\le 1+\deg v-\Delta},$$
which is the probability that the process ends with the set $M'=\mu_n(\psi_n)$ containing at most $1+\deg v-\Delta$ edges.
If that is the case, then since the set $M$ is $(v,\Delta,\beta)$-rich by (b), it follows that there are at least $\beta$ indices~$i$ for which the subset selector $f_i$
returns a non-empty set during the process.  For each of these indices, we have $|\mu_i(\psi_i)|\ge |\mu_{i-1}(\psi_i)|+1$ with probability at least $p$ by Claim~\ref{cl__probincr},
and since $|\mu_n(\psi_n)|\le 1+\deg v-\Delta\le 1$, this can only occur at most once.  Hence,
$$\Pr{|\mu_n(\psi_n)|\le 1+\deg v-\Delta}\le (1-p)^\beta+\beta p(1-p)^{\beta-1}=(1-p+\beta p)\cdot (1-p)^{\beta-1}\le (1+\beta p)e^{-(\beta-1)p}=q,$$
confirming (\ref{eq__toprove}).
\end{proof}

In particular, we immediately get the following consequence.

\begin{lemma}\label{lemma__can_color_sparse}
Let $\Delta$ be a positive integer, let $(\DD,\CC,\{\WW_C\}_{C\in \CC})$ be a setting for a minimal $\Delta$-counterexample $(G,L,\Pi)$,
let $S$ be a transversal for this setting, and let $\psi$ be a uniformly random $(L,\Pi)$-coloring of $G[S]$.
If $\Delta\ge 10^6$, then for every vertex $v\in V(G)\setminus \bigl(S\cup\bigcup\DD\bigr)$
and every partial $(L,\Pi)$-coloring $\sigma$ of $G[S]$ such that $N(v)\cap \dom(\sigma)=\emptyset$, we have
$$\Pr{\exc_\psi(v)\le 0\mid \text{$\psi$ extends $\sigma$}}\le \frac{1}{4\Delta^2}.$$
\end{lemma}
\begin{proof}
Let $\gamma=\tfrac{1}{5}\Delta$; since $S$ is a transversal, each vertex of $G$ has at most $\gamma$ neighbors in $S$.
Let $M$ be the set of all non-edges of $G[N(v)\cap S]$; since $S$ is a transversal, we have $|M|\ge \tfrac{1}{4800}\Delta^2$. Let $\beta=\tfrac{1}{5000}\Delta^2$.
For every non-edge $xy\in M$, less than $|N(v)\cap S|\le \tfrac{1}{5}\Delta$ non-edges of $M$
are incident with each of $x$ and $y$.
Hence, the number of non-edges in $M$ incident neither with $x$ nor with $y$ is at least
$$|M|-\tfrac{2}{5}\Delta\ge \tfrac{1}{4800}\Delta^2-\tfrac{2}{5}\Delta\ge \beta,$$
and thus $M$ is $(v,\Delta,\beta)$-rich.

\noindent Let $\delta=\tfrac{\Delta}{5}$ and note that for every $u\in V(G)$, we have
$$|L(u)|\ge \Delta-1>\frac{\Delta+4\gamma+\delta-4}{2}.$$
By Lemma~\ref{lemma__can_free_vertex}, the uniform $(L,\Pi)$-coloring of $G[S]$ $(\Delta,M)$-fails $v$ with probability at most
$$e\cdot \Bigl(1+\frac{\beta\delta}{\Delta^2}\Bigr)\cdot \exp\Bigl(-\frac{\beta\delta}{\Delta^2}\Bigr)=
  e\cdot \bigl(1+\tfrac{1}{25000}\Delta\bigr)\cdot \exp\bigl(-\tfrac{1}{25000}\Delta\bigr)\le \frac{1}{4\Delta^2},$$
and thus 
\[\Pr{\exc_\psi(v)\le 0\mid\text{$\psi$ extends $\sigma$}}\le \frac{1}{4\Delta^2}.\qedhere\]
\end{proof}

Let $\Delta$ be a positive integer, let $(\DD,\CC,\{\WW_C\}_{C\in \CC})$ be a setting for a minimal $\Delta$-counterexample $(G,L,\Pi)$,
let $S$ be a transversal for this setting, and let $\psi$ be a uniformly random $(L,\Pi)$-coloring of~$G[S]$.
Next, we need to similarly argue that each core $C\in \CC$ almost surely contains enough (at least two) vertices with positive $\psi$-excess,
and thus $\psi$ can be extended to $C$ using Observation~\ref{obs__greedy}.  Using Lemma~\ref{lemma__can_free_vertex}, it is easy to see
that each element of $\WW_C$ $C$-bisected by $S$ contributes such a vertex with a constant probability.
Since the events for distinct elements of $\WW_C$ are not independent, we need to use another resampling argument to
obtain the desired result.

More precisely, let $\varphi$ be an $(L,\Pi)$-coloring of $G[S]$, let $X$ be an element of $\WW_C$ which is $C$-bisected by $S$, and let $v$ be the vertex in $X\cap C$.
A \emph{$(C,\varphi)$-saving pair} for the set $X$ is a pair $(a,x)$ of vertices of $G$ such that
$a\in C\cap S$, $x\in X\setminus\{v\}$, and the colors $\varphi(a)$ and $\varphi(x)$ together block at most one color at $v$.
If there exists such a $(C,\varphi)$-saving pair, we say that the vertex $a$ is a \emph{$(C,\varphi)$-helper} for $X$.
We say that the set $X$ is \emph{$(C,\varphi)$-successful} if there exist at least $|X|-1$ distinct $(C,\varphi)$-saving pairs for $X$.
In that case, there exists at least one set $A\subseteq C\cap S$ such that
\begin{itemize}
\item the vertices of $A$ are $(C,\varphi)$-helpers for $X$, and 
\item there exist are at least $|X|-1$ distinct $(C,\varphi)$-saving pairs $(a,x)$ for $X$ such that $a\in A$;
\end{itemize}
we say that such a set $A$ \emph{$(C,\varphi)$-saves} $X$.
For a set $\WW'\subseteq \WW_C$, we define $\numsucc_C(\WW',\varphi)$ as the number of elements of $\WW'$
which are $C$-bisected by $S$ and $(C,\varphi)$-successful.

We are going to resample a uniformly random $(L,\Pi)$-coloring of $G[S]$ on the sets $(X\setminus C)\cup (C\cap S)$ for each element $X\in \WW_C$
which is $C$-bisected by $S$ one by one, arguing that there is a good probability~$X$ becomes successful.
However, as in the proof of Lemma~\ref{lemma__can_free_vertex}, we need to take care not to cancel previously processed successful sets,
and thus (by the means of a suitably defined subset selector $f$), we avoid resampling on the part of $C\cap S$ responsible for their success
(this motivates the definition of a \emph{witness set} below).  We also need to ensure that the resampling is $f$-invariant,
which turns out to be equivalent to removing certain colors (whose usage would change either the success status or the witness set
of a previously processed element of $\WW_C$) from the lists of some of the vertices (this motivates the definition of the \emph{witness protection zone}
and the \emph{witness-forbidden colors}).

Let us now formally state the required definitions. Let $\prec$ be an arbitrary fixed ordering of the vertices of $G$,
and let $\prec_l$ be the corresponding lexicographic ordering of the subsets of $V(G)$; that is, we have $A\prec_l A'$ if $A'\neq\emptyset$ and
\begin{itemize}
\item $A=\emptyset$, or
\item $A\neq\emptyset$ and $\min_\prec A\prec\min_\prec A'$, or
\item $A\neq\emptyset$, $\min_\prec A=\min_\prec A'$, and $A\setminus \{\min_\prec A\}\prec_l A'\setminus \{\min_\prec A'\}$.
\end{itemize}
Consider any element $X$ of $\WW_C$ which is $C$-bisected by $S$, and let $v$ be the unique vertex in $X\cap C$.
As before, let $\varphi$ be an $(L,\Pi)$-coloring of $G[S]$.
\begin{itemize}
\item If the set $X$ is $(C,\varphi)$-successful, then the \emph{$(C,\varphi)$-witness} for $X$ is the minimal
set $A_X\subseteq C\cap S$ (in the ordering $\prec_l$) that $(C,\varphi)$-saves $X$.  Note that $|A_X|\le |X|-1\le 2$.
The \emph{$(C,\varphi)$-witness protection zone} for $X$ consists of all vertices $b\in (C\cap S)\setminus A_X$ such that $b\prec \max_{\prec} A_X$.
\item If $X$ is not $(C,\varphi)$-successful, then the \emph{$(C,\varphi)$-witness} for $X$ is the set $A_X\subseteq C\cap S$ consisting of all $(C,\varphi)$-helpers for $X$.
Since $X$ is not $(C,\varphi)$-successful, we have $|A_X|<|X|-1\le 2$.
In this case, the \emph{$(C,\varphi)$-witness protection zone} for $X$ is equal to $(C\cap S)\setminus A_X$.
\end{itemize}
In both cases, for a vertex $b$ in the $(C,\varphi)$-witness protection zone for $X$, we say that a color $\alpha\in L(b)$ is \emph{$(C,\varphi,X)$-witness-forbidden} at $b$
if there exists a vertex $x\in X\setminus C$ such that $\varphi(x)$ and the color $\alpha$ at $b$ block the same color at $v$.  We need two observations.
\begin{itemize}
\item The correspondence assignment $(L,\Pi)$ is full by an assumption from the definition of a $\Delta$-counterexample,
and thus for each vertex $x\in X\setminus\{v\}$, there exists exactly one such color $\alpha\in L(b)$.
Consequently, there are at most $|X|-1\le 2$ colors that are $(C,\varphi,X)$-witness-forbidden at $b$.
\item Since $b$ is in the $(C,\varphi)$-witness protection zone, it is not a $(C,\varphi)$-helper for $X$ (when $X$ is $(C,\varphi)$-successful, this is by the minimality of $A_X$
in the lexicographic ordering $\prec_l$).  Moreover, recoloring $b$ turns it into a helper for $X$ if and only if the new color of $b$ is $(C,\varphi,X)$-witness-forbidden.
\end{itemize}
Based on these definitions, we can now define the desired subset selector and state the conditions
describing the corresponding invariant resampling.  For a set $\WW'\subseteq \WW_C$ containing only
elements which are $C$-bisected by $S=\dom(\varphi)$ and another such element $Y\in \WW_C\setminus \WW'$,
we let
$$f_{C,\WW',Y}(\varphi)=(Y\setminus C)\cup \Bigl((C\cap S)\setminus \bigcup_{X\in \WW'} A_X\Bigr).$$
Let us also define a correspondence assignment $(L_{C,\WW',Y,\varphi},\Pi_{C,\WW',Y,\varphi})$ as follows.
We start from the correspondence assignment $(L,\Pi)$, and for each set $X\in \WW'$ and each vertex $b\in f_{C,\WW',Y}(\varphi)$ belonging to the
$(C,\varphi)$-witness protection zone for $X$, we remove all (at most two) $(C,\varphi,X)$-witness-forbidden colors at $b$ from the list of $b$.
When a color is removed, we also remove the incident correspondences.

\begin{lemma}\label{lemma__invariant}
Let $\Delta$ be a positive integer, let $(\DD,\CC,\{\WW_C\}_{C\in \CC})$ be a setting for a minimal $\Delta$-counterexample $(G,L,\Pi)$,
let $S$ be a transversal for this setting, and let $\varphi$ be an $(L,\Pi)$-coloring of $G[S]$.
Consider a core $C\in\CC$, a set $\WW'\subseteq \WW_C$ containing only elements which are $C$-bisected by $S$,
and another such element $Y\in \WW_C\setminus \WW'$, and let $f=f_{C,\WW',Y}$ and $(L',\Pi')=(L_{C,\WW',Y,\varphi},\Pi_{C,\WW',Y,\varphi})$.
An $f(\varphi)$-modification $\varphi'$ of $\varphi$ is $f$-invariant if and only if $\varphi'$ is an $(L',\Pi')$-coloring.
\end{lemma}
\begin{proof}
Let us fix an $f(\varphi)$-modification $\varphi'$ of $\varphi$, and
suppose first that $\varphi'$ is an $(L',\Pi')$-coloring.  Consider any set $X\in\WW'$.
\begin{itemize}
\item If $X$ is $(C,\varphi)$-successful, then this is ensured already by the colors of vertices in the $(C,\varphi)$-witness $A_X\subseteq (C\cap S)\setminus f(\varphi)$,
and the colors of these vertices are the same in $\varphi'$.
Thus, $X$ is also $(C,\varphi')$-sucessful.  Moreover, the choice of $L'$ ensures that no vertex $b\in (C\cap S)\setminus A_X$ such that $b\prec \max_\prec A_X$
is a $(C,\varphi')$-helper, and thus $A_X$ is also the $(C,\varphi')$-witness.
\item If $X$ is not $(C,\varphi)$-successful, then the $(C,\varphi)$-witness $A_X\subseteq (C\cap S)\setminus f(\varphi)$ contains all $(C,\varphi)$-helpers for $X$.
The coloring $\varphi'$ assigns the same colors to these vertices, and thus they are also $(C,\varphi')$-helpers.
Moreover, the choice of $L'$ ensures that no vertex in $(C\cap S)\setminus A_X$ is a $(C,\varphi')$-helper; consequently, $A_X$ is also the $(C,\varphi')$-witness.
\end{itemize}
Since all elements of $\WW'$ have their $(C,\varphi')$-witness equal to their $(C,\varphi)$-witness, it follows that $f(\varphi')=f(\varphi)$, and thus
$\varphi'$ is an $f$-invariant modification of $\varphi$.

Conversely, assume that $\varphi'$ is an $f$-invariant modification of $\varphi$, and suppose for a contradiction that $\varphi'$ is not an $(L',\Pi')$-coloring.
Let $b\in C\cap S$ be the smallest vertex (in the ordering $\prec$) such that $\varphi'(b)\not\in L'(b)$.
Since $L'(b)\neq L(b)$, we have $b\in f(\varphi)$; and moreover, since $\varphi'(b)\not\in L'(b)$,
there exists a set $X\in \WW'$ such that $b$ is in the $(C,\varphi)$-witness protection zone for $X$ and the color $\varphi'(b)$ is $(C,\varphi,X)$-witness-forbidden.
Hence, $b$ is a $(C,\varphi')$-helper for $X$.  Since $b\in f(\varphi)=f(\varphi')$, the vertex $b$ does not belong to the $(C,\varphi')$-witness $A'_X$ for $X$.
This means that $X$ is $(C,\varphi')$-successful, and by the lexicographical minimality of $A'_X$, we have $b\succ\max_\prec A'_X$.

Consider any vertex $b'\in A'_X$.  If $b'\not\in A_X$, then since $b$ is in the $(C,\varphi)$-witness protection zone for $X$ and $b'\prec b$,
the vertex $b'$ belongs to the $(C,\varphi)$-witness protection zone for $X$.  Moreover, by the minimality of $b$, we have $\varphi'(b')\in L'(b)$,
and thus the color $\varphi'(b')$ is not $(C,\varphi,X)$-witness-forbidden.  However, this implies that $b'$ is not a $(C,\varphi')$-helper for $X$,
contradicting the claim that $b'\in A'_X$.  Therefore, we have $b'\in A_X$ for every $b'\in A'_X$, i.e., $A'_X\subseteq A_X$.

It follows that $\varphi'$ assigns the same colors to the vertices of $A'_X$ as the coloring $\varphi$; and since already the colors of vertices
in $A'_X$ show that $X$ is $(C,\varphi')$-successful, the set $X$ must also be $(C,\varphi)$-successful.  Moreover, by the lexicographic minimality of $A_X$,
we have $A_X=A'_X$.  However, since $b\succ\max_\prec A'_X$, this would imply that $b$ is not the $(C,\varphi)$-witness protection zone for $X$, which is a contradiction.

It follows that if $\varphi'$ is an $f$-invariant modification of $\varphi$, then it is an $(L',\Pi')$-coloring.
\end{proof}

We are now ready to deal with the cores of the sparse-dense decomposition.

\begin{lemma}\label{lemma__can_color_core}
Let $\Delta$ be a positive integer, let $(\DD,\CC,\{\WW_C\}_{C\in \CC})$ be a setting for a minimal $\Delta$-counterexample $(G,L,\Pi)$,
let $S$ be a transversal for this setting, and let $\psi$ be a uniformly random $(L,\Pi)$-coloring of $G[S]$.
If $\Delta\ge \thebound$, then for every core $C\in \CC$ and every partial $(L,\Pi)$-coloring $\sigma$ of $G[S]$ such that
$\bigl(C\cup \bigcup \WW_C\bigr)\cap \dom(\sigma)=\emptyset$, we have
$$\Pr{\numsucc_C(\WW_C,\psi)\le 1\mid \text{$\psi$ extends $\sigma$}}\le \frac{1}{4\Delta^2}.$$
\end{lemma}
\begin{proof}
Let $m=\lceil\tfrac{1}{10^4}\Delta\rceil$.  Since $S$ is a transversal, there exist distinct elements $X_1,\ldots,X_m\in\WW_C$ which are $C$-bisected by $S$.
Let $\WW_0 = \emptyset$ and for each $i\in[m]$, let $\WW_i=\{X_1,\ldots, X_i\}$.
For each $i\in [m]$, let $v_i$ be the unique vertex in $X_i\cap C$, let $f_i$ denote the subset selector $f_{C,\WW_{i-1},X_i}$, and
for any $(L,\Pi)$-coloring $\varphi$ of $G[S]$, let $(L_{\varphi,i},\Pi_{\varphi,i})$ denote the correspondence assignment
$(L_{C,\WW_{i-1},X_i,\varphi},\Pi_{C,\WW_{i-1},X_i,\varphi})$.

By Observation~\ref{obs__combine} and Lemmas~\ref{tool__resampling} and \ref{lemma__invariant}, we can obtain a uniformly random $(L,\Pi)$-coloring $\psi$ of $G[S]$ extending $\sigma$
by the following process.  Let $\psi_0$ be a uniformly random $(L,\Pi)$-coloring of $G[S]$ extending $\sigma$.
For $i\in\{1,\ldots,m\}$, let $\psi_i$ be an $f_i$-invariant resampling of $\psi_{i-1}$;
that is, letting $\sigma_{i-1}$ be the restriction of $\psi_{i-1}$ to $S\setminus f_i(\psi_{i-1})$,
$\psi_i$ is a uniformly random $(L_{\psi_{i-1},i},\Pi_{\psi_{i-1},i})$-coloring of $G[S]$ extending $\sigma_{i-1}$.
Finally, let $\psi=\psi_m$.  Let us note that for each $i\in [m]$,
\begin{itemize}
\item since each $(C,\psi_{i-1})$-witness of an element of $\WW_{i-1}$ has size at most two and $S$ is a transversal, we have $|C\cap f(\psi_{i-1})|\ge |C\cap S|-2(m-1)\ge\tfrac{1}{6}\Delta-\tfrac{1}{5000}\Delta$, and
\item since each element $X\in \WW_{i-1}$ contributes at most two $(C,\psi_{i-1},X)$-witness-forbidden colors at each vertex $u\in f_i(\psi_{i-1})$, we have
$|L_{\psi_{i-1},i}(u)|\ge \Delta-1-2(m-1)\ge \Delta-\tfrac{1}{5000}\Delta-1$.
\end{itemize}

The choice of the subset selectors $f_i$ implies that $\numsucc_C(\WW_i,\psi_i)\ge \numsucc_C(\WW_{i-1},\psi_{i-1})$.
Moreover, if the set $X_i$ is $(C,\psi_i)$-successful, then $\numsucc_C(\WW_i,\psi_i)\ge \numsucc_C(\WW_{i-1},\psi_{i-1})+1$.
\begin{claim}\label{cl__probsucc}
For each $i\in [m]$, the probability that the set $X_i$ is $(C,\psi_i)$-successful is at least $\tfrac{1}{5000}$.
\end{claim}
\begin{subproof}
Let $M_i$ be the set of all non-edges with one endpoint in $X_i\setminus \{v_i\}$ and the other endpoint in $C\cap f_i(\psi_{i-1})$.
Observe that if the coloring $\psi_i$ $(\Delta,M_i)$-saves the vertex $v_i$, then the set $X_i$ is $(C,\psi_i)$-successful.
Thus, it suffices to use Lemma~\ref{lemma__can_free_vertex} to give a lower bound on the probability that $\psi_i$, a uniformly random $(L_{\psi_{i-1},i},\Pi_{\psi_{i-1},i})$-coloring of $G[S]$ extending $\sigma_{i-1}$,
$(\Delta,M_i)$-saves the vertex $v_i$.
\begin{itemize}
\item Since $S$ is a transversal, each vertex of $G$ has at most $\gamma=\tfrac{1}{5}\Delta$ neighbors in $S$.
\item Let $\beta=\tfrac{1}{7}\Delta$.
Since $X_i$ is a $(\Delta,85)$-benefactor for $C$, we have $2+\deg v_i-\Delta=|X_i\setminus\{v_i\}|$
and each vertex in $X_i\setminus\{v_i\}$ is incident with at least
$$|C\cap f_i(\psi_{i-1})|-85\ge \tfrac{1}{6}\Delta-\tfrac{1}{5000}\Delta-85\ge \beta+1$$
non-edges from $M_i$.  Consequently, the set $M_i$ is $(v,\Delta,\beta)$-rich.
\item Let $\delta=\tfrac{1}{6}\Delta$.  For each vertex $u\in N(v_i)\cap S$, we have
$$|L_{\psi_{i-1},i}(u)|\ge \Delta-\tfrac{1}{5000}\Delta-1\ge \tfrac{\Delta+4\gamma+\delta-4}{2}.$$
\end{itemize}
Let $p=\tfrac{\delta}{\Delta^2}=\tfrac{1}{6\Delta}$ and note that $\beta p=\tfrac{1}{42}$.
By Lemma~\ref{lemma__can_free_vertex}, the set $X_i$ is not $(C,\psi_i)$-successful
with probability at most 
$$(1+\beta p)e^{-(\beta-1)p}=\tfrac{43}{42}\cdot e^{-\frac{1}{42}}\cdot e^{\frac{1}{6\Delta}}\le 1-\tfrac{1}{5000},$$
as required.
\end{subproof}

Let $q=\tfrac{1}{5000}$.  The probability that $\numsucc_C(\WW_C,\psi)\le 1$ is bounded from above by the probability that
the set $X_i$ is $(C,\psi_i)$-successful for at most one index $i\in [m]$, which by Claim~\ref{cl__probsucc} is at most
\begin{align*}
(1-q)^m+mq(1-q)^{m-1}&\le (1+qm)\cdot (1-q)^{m-1}\le \frac{1+qm}{1-q}e^{-qm}\\
&\le 2q(m-1) e^{-qm}\le \frac{\Delta}{25\cdot 10^6}\cdot e^{-\frac{\Delta}{5\cdot 10^7}}\le \frac{1}{4\Delta^2}. \qedhere
\end{align*}
\end{proof}

\subsection{Putting things together}

We are now ready to finish the proof of our main result.

\begin{proof}[Proof of Theorem~\ref{thm__bk_conjecture_is_true_asym_for_DP_coloring}]
Let $\Delta\ge \thebound$ be an integer, and suppose for a contradiction that there exists a graph of maximum degree at most $\Delta$ and clique number at most $\Delta-1$
which is not $(\Delta-1)$-correspondence colorable.  Thus, there exists a minimal $\Delta$-counterexample $(G,L,\Pi)$.
By Corollary~\ref{cor__summary}, there exists a setting $(\DD,\CC,\{\WW_C\}_{C\in \CC})$ for this counterexample,
and by Lemma~\ref{lemma__extransvers}, there exists a transversal~$S$ for this setting.
Let $\psi$ be a uniformly random $(L,\Pi)$-coloring of $G[S]$.
\begin{itemize}
\item For each vertex $v\in V(G)\setminus \bigl(S\cup \bigcup \DD\bigr)$, let $A_v$ be the event that $\exc_\psi(v)\le 0$, and let $B_v$ be the set consisting of the neighbors of $v$ in $S$.
By Lemma~\ref{lemma__can_color_sparse}, for every $(L,\Pi)$-coloring $\sigma$ of $G[S\setminus B_v]$, we have
$$\Pr{A_v\mid \text{$\psi$ extends $\sigma$}}\le \frac{1}{4\Delta^2}.$$
Moreover, we have $|B_v|\le \tfrac{1}{5}\Delta$.
\item For each core $C\in \CC$, let $A_C$ be the event that $\numsucc_C(\WW_C,\psi)\le 1$ and let $B_C=\bigl(C\cup \bigcup\WW_C\bigr)\cap S$.
By Lemma~\ref{lemma__can_color_core}, for every $(L,\Pi)$-coloring $\sigma$ of $G[S\setminus B_C]$, we have
$$\Pr{A_C\mid \text{$\psi$ extends $\sigma$}}\le \frac{1}{4\Delta^2}.$$
Moreover, we have $|B_C|\le |C\cap S|+2|\WW_C|\le \tfrac{1}{5}\Delta+2\cdot\lceil \tfrac{1}{100}\Delta\rceil\le \tfrac{1}{4}\Delta$.
Let us remark that the bound $|C\cap S|\le \tfrac{1}{5}\Delta$ holds since $S$ is a transversal and $C\cap S\subseteq N(z)$ for any vertex $z\in C\setminus S$.
\end{itemize}
Let $I=\bigl(V(G)\setminus \bigl(S\cup \bigcup \DD\bigr)\bigr)\cup \CC$ and let $\EE=\{A_x:x\in I\}$ be the set of all events discussed above.
For each $x\in I$, let $\EE_{A_x}=\{A_y:y\in I,B_y\cap B_x\neq\emptyset\}$.
Note that for each $x\in I$, we have $|\EE_{A_x}|<\Delta^2$; indeed, the set $\EE_{A_x}$ contains
\begin{itemize}
\item events $A_v$ for at most $|B_x|\cdot \Delta\le \tfrac{1}{4}\Delta^2$ vertices $v\in V(G)\setminus \bigl(S\cup \bigcup \DD\bigr)$ with a neighbor in $B_x$, and
\item events $A_C$ for at most $|B_x|+|B_x|\cdot \Delta\le \tfrac{1}{4}\Delta(\Delta+1)$ cores $C\in\CC$ such that
either a vertex of $B_x$ or a neighbor of a vertex of $B_x$ belongs to $C$.
\end{itemize}
Consider any set $\EE'\subseteq \EE\setminus \EE_{A_x}$ and note that the validity of the events in $\EE'$ only depends on the colors of the vertices in $S\setminus B_x$.
Let $\Phi$ be the set of all $(L,\Pi)$-colorings $\sigma$ of $G[S\setminus B_x]$ for which none of the events in $\EE'$ occur for the $(L,\Pi)$-colorings of $G[S]$ that extend $\sigma$.
Thus
$$\sum_{\sigma\in \Phi} \Pr{\text{$\psi$ extends $\sigma$}\Bigm|\bigwedge_{A\in \EE'} \lnot A}=1,$$
and
\begin{align*}
\Pr{A_x\Bigm|\bigwedge_{y\in \EE'} \lnot A_y}&=\sum_{\sigma\in \Phi} \Pr{A_x\mid\text{$\psi$ extends $\sigma$}}\cdot \Pr{\text{$\psi$ extends $\sigma$}\Bigm|\bigwedge_{y\in \EE'} \lnot A_y}\\
   &\le  \frac{1}{4\Delta^2} \cdot \sum_{\sigma\in \Phi} \Pr{\text{$\psi$ extends $\sigma$}\Bigm|\bigwedge_{A\in \EE'} \lnot A}= \frac{1}{4\Delta^2}.
\end{align*}
By  the Lovász Local Lemma (Theorem~\ref{tool__lovasz_local_lemma}), with non-zero probability none of the events in $\EE$ holds.
Let us fix an $(L,\Pi)$-coloring~$\psi$ of $G[S]$ with this property.
It is now easy to see that $\psi$ extends to an $(L,\Pi)$-coloring of $G$.
Indeed, let $\{X_1, \ldots, X_t\}$ be the partition of $V(G)\setminus S$ such that
for each $i\in[t]$, either $X_i=\{v\}$ for a vertex $v\in V(G)\setminus \bigl(S\cup \bigcup \DD\bigr)$,
or $X_i=D\setminus S$ for a near-clique $D\in \DD$.
Let $\varphi_0=\psi$.  For each $i\in [t]$, we extend $\varphi_{i-1}$ to a partial $(L,\Pi)$-coloring $\varphi_i$ by giving colors to the vertices in $X_i$
as follows:
\begin{itemize}
\item If $X_i=\{v\}$ for a vertex $v\in V(G)\setminus \bigl(S\cup \bigcup \DD\bigr)$, then since $A_v$ is false,
we have $\exc_{\varphi_{i-1}}(v)\ge \exc_{\psi}(v)>0$.  Therefore, we can extend $\varphi_{i-1}$ to the vertex $v$ greedily.
\item Suppose now that $X_i=D\setminus S$ for a near-clique $D\in\DD$ with core $C\in \CC$.  Recall that $\delta(G[D])\ge \tfrac{3}{4}\Delta$.
If there exists a vertex $u\in D\setminus (C\cup S)$,
then let $\varphi'_{i-1}$ be a partial $(L,\Pi)$ coloring obtained from $\varphi_{i-1}$ by coloring $u$ greedily.  This is possible,
since $u$ has at least $\tfrac{3}{4}\Delta-|N(u)\cap S|\ge \tfrac{3}{4}\Delta-\tfrac{1}{5}\Delta\ge 2$ neighbors in $C\setminus S$ that
have not been colored yet.  Otherwise, let $\varphi'_{i-1}=\varphi_{i-1}$.

Since $A_C$ is false, the set $C\setminus S$ contains distinct vertices $v_1$ and $v_2$ with positive $\psi$-excess,
and thus also with positive $\varphi'_{i-1}$-excess. Therefore, we can extend $\varphi'_{i-1}$ to $C\setminus S$ by Observation~\ref{obs__greedy}, coloring $v_1$ and $v_2$ last.
\end{itemize}
This way, we obtain an $(L,\Pi)$-coloring $\varphi_t$ of $G$.  This contradicts the claim that $(G,L,\Pi)$ is a $\Delta$-counterexample.
\end{proof}

%===================================================================

\subsubsection*{Open access statement}%\mbox{}\\*
For the purpose of open access, a CC BY public copyright license is applied to any Author Accepted Manuscript (AAM) arising from this submission.

\end{document}